\documentclass[10pt]{article}
\usepackage{amssymb,amsmath,amsfonts,amsthm,enumerate,color}
\usepackage[toc,page,title,titletoc,header]{appendix}
\usepackage{graphicx}
\usepackage{indentfirst}
\usepackage{multicol}
\usepackage{booktabs}

\numberwithin{equation}{section}

\newtheorem{Theorem}{Theorem}[section]
\newtheorem{Lemma}[Theorem]{Lemma}

\newtheorem{Definition}[Theorem]{Definition}
\newtheorem{Corollary}[Theorem]{Corollary}

\numberwithin{equation}{section}

 \def\p{\partial} \def\nb{\nonumber}
\def \Vh0{\stackrel{\circ}{V}_h} \def\to{\rightarrow}
   
\def\Om{\Omega}  
\def\I{ {\rm (I) } }  \def\II{ {\rm (II) } }
\newcommand{\q}{\quad}    
\def\l{\label}  \def\f{\frac}  

\def\m{\mbox}   
\def\ms{\medskip}  

\def\p{\partial}

\newcommand{\lc}
{\mathrel{\raise2pt\hbox{${\mathop<\limits_{\raise1pt\hbox
{\mbox{$\sim$}}}}$}}}

\newcommand{\gc}
{\mathrel{\raise2pt\hbox{${\mathop>\limits_{\raise1pt\hbox{\mbox{$\sim$}}}}$}}}

\newcommand{\ec}
{\mathrel{\raise2pt\hbox{${\mathop=\limits_{\raise1pt\hbox{\mbox{$\sim$}}}}$}}}

\def\bb{\begin{equation}} \def\ee{\end{equation}}

\def\beqn{\begin{eqnarray}}  \def\eqn{\end{eqnarray}}

\def\beqnx{\begin{eqnarray*}} \def\eqnx{\end{eqnarray*}}

\def\bn{\begin{enumerate}} \def\en{\end{enumerate}}

\def\bd{\begin{description}} \def\ed{\end{description}}



\title{The Concept of Heterogeneous Scattering Coefficients and Its Application in Inverse Medium Scattering}

\author{
Habib Ammari\footnote{Department of Mathematics and Applications, Ecole Normale Sup$\acute{\text{e}}$rieure, 45 Rue d'Ulm, 75005 Paris, France. The work of this author
was supported by ERC Advanced Grant Project MULTIMOD--267184.
(habib.ammari@ens.fr).}
\and Yat Tin Chow\footnote{Department of Mathematics, Chinese University of Hong Kong, Shatin, N.T., Hong Kong (ytchow@math.cuhk.edu.hk).}
\and Jun Zou\footnote{Department of Mathematics, Chinese University of Hong Kong, Shatin, N.T., Hong Kong.
The work of this author was substantially supported by Hong Kong RGC grants (projects 405513 and 404611).
(zou@math.cuhk.edu.hk).}
}

\begin{document}

\date{}
\maketitle

\begin{abstract}
This work investigates the scattering coefficients for inverse medium scattering problems. It shows some fundamental properties of the coefficients such as symmetry and tensorial properties.
The relationship between the scattering coefficients and the far-field pattern is also derived.
Furthermore,  the sensitivity of the scattering coefficients with respect to changes in the permittivity and permeability distributions is investigated. In the linearized case, explicit formulas for reconstructing permittivity and permeability distributions from the scattering coefficients is proposed.  They relate the exponentially ill-posed character of the inverse medium scattering problem at a fixed frequency to the exponential decay of the scattering coefficients. Moreover, they show the stability of the reconstruction from multifrequency measurements. This provides a new direction for solving inverse medium scattering problems.
\end{abstract}

\bigskip

\noindent {\footnotesize Mathematics Subject Classification
(MSC2000): 35R30, 35B30}


\section{Introduction}
In this work we will be concerned with the following transverse magnetic polarized wave scattering problem
\beqn
    \nabla \cdot \f{1}{\mu} \nabla u + \omega^2 \varepsilon u = 0  & \text{ in } \mathbb{R}^2 \, ,
\label{scattering1}
\eqn
where $\mu, \varepsilon > 0 $ are the respective permittivity and permeability coefficients of the medium.  We consider an inhomogeneous medium $\Omega$ contained inside a homogeneous background medium, and assume that $\Omega$ is an open bounded connected domain with a $\mathcal{C}^{1,\alpha}$  boundary for some $0< \alpha<1$. Let $\nu$ denote the outward normal vector at $\partial \Omega$, and $\mu_0, \varepsilon_0 > 0 $ be the medium coefficients of the homogeneous background medium.  Suppose that
$\mu, \varepsilon \in L^\infty$ and $\mu - \mu_0$ and $\varepsilon - \varepsilon_0$ are supported in $\Omega$. Moreover, there exist positive constants $\underline{\mu}$ and $\underline{\varepsilon}$ such that $\mu(x)\geq \underline{\mu}$ and $\varepsilon(x) \geq \underline{\varepsilon}$ in $\Omega$. Under these settings, we can write the equation (\ref{scattering1}) as follows:
\beqn
\begin{cases}
    \nabla \cdot \f{1}{\mu} \nabla u + \omega^2 \varepsilon u = 0 & \text{ in } \Omega \, , \\
    \Delta u + k_0^2 u = 0 & \text{ in } \mathbb{R}^2 \backslash \overline{\Omega} , \\
    u^+  = u^{-} & \text{ on } \partial \Omega \, , \\
    \f{1}{\mu_0} \f{\partial u^{+}}{\partial \nu} = \f{1}{\mu} \f{\partial u^{-}}{\partial \nu} & \text{ on } \partial \Omega \, . \\
\end{cases}
\label{scattering2}
\eqn
Here and throughout this paper, the superscripts $\pm$ indicate the limits from outside and inside of $\Omega$, respectively, and $\partial/\partial \nu$ denotes the normal derivative.
We shall complement the system (\ref{scattering2}) by the physical outgoing Sommerfeld radiation condition:
\beqn
    \f{\partial}{\partial r} (u - u_0) - i k_0 (u - u_0) = O(|x|^{-\f{3}{2}})& \text{ as } |x| \rightarrow \infty \, .
    \label{sommerfield}
\eqn
where  $k_0 = \omega \sqrt{\mu_0 \varepsilon_0} $ is the wavenumber and $u_0$ is an incident field, solving the homogeneous Helmholtz equation $ (\Delta + k_0^2) u_0 = 0$
in $\mathbb{R}^d$.
The solution $u$ to the system (\ref{scattering2}) and (\ref{sommerfield})
represents the total field due to the scattering from the inclusion $\Omega$ corresponding to the incident field $u_0$.

The notion of scattering coefficients was previously studied
for homogeneous electromagnetic inclusions \cite{homoscattering} (see also \cite{lim}) in order to enhance near-cloaking.
The purpose of this paper is twofold.  We first introduce the concept of inhomogeneous
scattering coefficients  and investigate some of their important properties and their
sensitivity with respect to  changes in the physical parameters. Then we make use of this new concept for solving the inverse medium scattering  problem and understanding the associated fundamental issues of stability and resolution. The inhomogeneous scattering coefficients can be obtained from the far-field data by a least-squares method \cite{tran}. 
Explicit reconstruction formulas of the inhomogeneous electromagnetic parameters from the scattering coefficients at a fixed frequency or at multiple frequencies are derived in the linearized case. These formulas show that the  exponentially ill-posed 
characteristics of the inverse medium scattering problem at a fixed frequency \cite{gallagher, isakov1, john} is due to the exponential decay of the scattering coefficients. Moreover, they clearly indicate the stability of the reconstruction from multifrequency measurements \cite{bao2,bao3, isakov2, isakov3}. Based on the decay property of the inhomogeneous scattering coefficients, a resolution analysis analogous to the one in \cite{foundations} can be easily derived. The resolving power, i.e., the number of scattering coefficients which can be stably reconstructed from the far-field measurements,  can be expressed in terms of the signal-to-noise ratio in the far-field measurements.  The scattering coefficient based approach introduced in this paper is a new promising direction for solving the long-standing inverse scattering problem with heterogeneous inclusions. It could be combined with the continuation method developed in \cite{bao1,coifman} for achieving a good resolution and stability for the image reconstruction.

For the sake of simplicity, we shall restrict ourselves to the scattering problem in two dimensions,
but all the results and analysis hold true also for three dimensions.

The paper is organized as follows. In section \ref{sec2} we introduce the notion of inhomogeneous scattering coefficients. Section \ref{sec3} provides integral representations of the scattering coefficients and shows their exponential decay. This property is  the root cause of the exponentially ill-posed character of the inverse medium scattering problem. In section \ref{sec4} we prove that the scattering coefficients are nothing else but the Fourier coefficients of the far-field pattern, then derive transformation formulas for the scattering coefficients under rigid transformations and scaling in section \ref{sec5}. In section \ref{sec6} we provide a sensitivity analysis with respect to the electromagnetic parameters for the scattering coefficients. In section \ref{sec7} we derive new reconstruction formulas from the scattering coefficients at one frequency and at multiple frequencies as well. A few concluding remarks are given in section \ref{sec8}. Appendix \ref{appendixA} is to construct a Neumann function for the inhomogeneous Helmholtz equation on a bounded domain. Appendices \ref{appendixB} and \ref{appendixC} are to show the existence of some functions  used in the derivation of the explicit reconstruction formulas in the linearized case.

\section{Integral representation and scattering coefficients} \label{sec2}
In this section we define the scattering coefficients of inhomogeneous inclusions. The idea of scattering coefficients for inclusions with homogeneous permittivity and permeability was initially introduced in \cite{homoscattering}. We extend this idea and define such a notion for inhomogeneous inclusions following the idea in \cite{nonhomopolarization,homoscattering}. We first derive the fundamental representation of the solution
$u$ to the system (\ref{scattering2})-(\ref{sommerfield}). For $k_0>0$, let $\Phi_{k_0}$ be the fundamental solution to the Helmholtz operator $\Delta + k_0^2$ in two dimensions satisfying
\[
    (\Delta + k_0^2)\Phi_{k_0}(x) = \delta_0(x)
\]
subject to the outgoing Sommerfeld radiation condition:
$$
    \f{\partial}{\partial r}\Phi_{k_0} - i k_0 \Phi_{k_0} = O(|x|^{-\f{3}{2}}) \quad \text{ as } |x| \rightarrow \infty \, .
$$
  Then $\Phi_{k_0}$ is given by
\beqn
    \Phi_{k_0} (x) = -\f{i}{4} H^{(1)}_0(k_0|x|) \, ,
    \label{fundamental}
\eqn
where $H^{(1)}_0$ is the Hankel function of the first kind of order zero.
We can easily deduce from Green's formula that if $u$ is the solution to (\ref{scattering2})-(\ref{sommerfield}), then we have for $x \in \mathbb{R}^2 \backslash \overline{\Omega}$ that
\beqn
    (u - u_0)(x) &=&
    \int_{\p \Omega}  \Phi_{k_0} (x-y) \f{\p (u- u_0)^{+}}{\p \nu}  (y) d \sigma(y) - \int_{\p \Omega}  \f{\p \Phi_{k_0} (x-y) }{\p \nu_y}  (u- u_0)^+ (y) d \sigma(y) \notag \\
    &=&
    \int_{\p \Omega}  \Phi_{k_0} (x-y) \f{\p u^{+}}{\p \nu}  (y) d \sigma(y) - \int_{\p \Omega}  \f{\p \Phi_{k_0} (x-y) }{\p \nu_y}  u^+ (y) d \sigma(y) \notag \\
    &=&
    \int_{\p \Omega}  \left(\f{\mu_0}{\mu}\right) \Phi_{k_0} (x-y)  \f{\p u^{-}}{\p \nu}  (y) d \sigma(y) - \int_{\p \Omega}  \f{\p \Phi_{k_0} (x-y) }{\p \nu_y}  u^- (y) d \sigma(y) \,, \label{scattered1}
\eqn
where the second equality holds since $u_0$ satisfies the homogeneous Helmholtz equation. Let
$g = \f{1}{\mu} \f{\p u^{-}}{\p \nu} $. Then we define the Neumann-to-Dirichlet (NtD) map
$\Lambda_{\mu,\varepsilon}$:
$H^{-\f{1}{2}}(\p \Omega) \to H^{\f{1}{2}}(\p \Omega)$ such that for any $g\in H^{-\f{1}{2}}(\p \Omega)$,
$u=\Lambda_{\mu,\varepsilon}g\in H^{\f{1}{2}}(\p \Omega)$ is the trace of the solution to the following system:
\beqn
    \nabla \cdot \f{1}{\mu} \nabla u + \omega^2 \varepsilon u = 0  \text{ in } \Omega \,; \quad
    \f{1}{\mu} \f{\partial u}{\partial \nu} = g  \text{ on } \partial \Omega \,.
    \label{interior_pde}
\eqn
We remark that $\Lambda_{\mu_0,\varepsilon_0}$ is well-defined if $\omega \sqrt{\mu_0 \varepsilon_0}$ is not a Neumann eigenvalue of $-\Delta$ on $\Omega$. For general distributions $\varepsilon$ and $\mu$, in order to ensure the well-posedeness of $\Lambda_{\mu,\varepsilon}$, one should assume, throughout this paper, that $0$ is not a Neumann eiganvalue of
$\nabla \cdot (1/\mu) \nabla + \omega^2 \varepsilon$ in $\Omega$.

 With this definition of $\Lambda_{\mu,\varepsilon}$,
we have $\Lambda_{\mu,\varepsilon} [g] = u^{-}$ and $\f{1}{\mu_0}\Lambda_{\mu_0,\varepsilon_0} [\f{\p\Phi_{k_0}}{\p \nu}] = \Phi_{k_0} $ on $\p \Om$.  We can therefore rewrite (\ref{scattered1}) as
\beqnx
    (u - u_0)(x) &=&
    \int_{\p \Omega}  \mu_0 \Phi_{k_0} (x-y) g(y) d \sigma(y) - \int_{\p \Omega}   \mu_0  \Lambda_{\mu_0,\varepsilon_0}^{-1} [\Phi_{k_0}](x-y)\Lambda_{\mu,\varepsilon} [g] (y) d \sigma(y) \,. \label{scattered2}
\eqnx
One can check that $\Lambda_{\mu,\varepsilon}$ is self-adjoint under the duality pair $\langle \cdot, \cdot \rangle_{H^{-\f{1}{2}},H^{\f{1}{2}}}$ on $\p \Omega$. So, we can further write
\beqn
    (u - u_0)(x) = \mu_0  \int_{\p \Omega} \Phi_{k_0} (x-y) \Lambda_{\mu_0,\varepsilon_0}^{-1} ( \Lambda_{\mu_0,\varepsilon_0} - \Lambda_{\mu,\varepsilon} ) [g] (y) d \sigma(y) \, , \quad x \in \mathbb{R}^d \backslash \overline{\Omega}.
    \label{ingegral1}
\eqn

We now use Graf's addition formula \cite{Watson} to derive an asymptotic
expression of $u - u_0$ as $|x| \rightarrow \infty$.
For the fundamental solution (\ref{fundamental}),
we recall the Graf's addition formula for $|x| > |y|$:
\beqn
    H^{(1)}_0 (k_0 |x-y|) = \sum_{n\in \mathbb{Z}} H^{(1)}_n (k_0 |x|) e^{i n \theta_x} J_n (k_0 |y|) e^{-i n \theta_y} \, ,
    \label{graf}
\eqn
where $x$ is in polar coordinate  $(|x|,\theta_x)$, and the same for $y$.
Now we define
\beqn
    (u_0)_m (y) := J_m (k_0 |y|) e^{i m \theta_y}\,, \label{incidence}
\eqn
and let $u_m$ to be the total field corresponding to the incident field $(u_0)_m$, namely
the solution to (\ref{scattering2})-(\ref{sommerfield}) with the incident field $u_0$ replaced
by $(u_0)_m$.  If we write
\beqn
    g_m := \f{1}{\mu} \f{\p u_m^{-}}{\p \nu} \,,
\eqn
then for any incident field $u_0$ admitting the expansion
\beqn
    u_0 (y) = \sum_{m\in \mathbb{Z}} a_m J_m (k_0 |y|) e^{i m \theta_y} \,,
\eqn
we have
\beqn
    g = \f{1}{\mu} \f{\,\,\,\p u^{-} }{\p \nu}  = \sum_{m\in \mathbb{Z}} a_m g_m \, .
    \label{gm}
\eqn
Putting (\ref{graf}) and (\ref{gm}) into (\ref{ingegral1}), we get the following asymptotic formula as $|x| \rightarrow \infty$:
\beqn
    & &(u - u_0)(x) \notag \\
    &=& -\f{i \mu_0}{4} \sum_{m,n\in \mathbb{Z}}  \int_{\p \Omega} a_m  H^{(1)}_n (k_0 |x|) J_n (k_0 |y|) e^{i n (\theta_x - \theta_y )} \Lambda_{\mu_0,\varepsilon_0}^{-1} ( \Lambda_{\mu_0,\varepsilon_0} - \Lambda_{\mu,\varepsilon} ) [g_m] (y) d \sigma(y) \,.
    \label{ingegral2}
\eqn
This motivates us to introduce the following definition.
\begin{Definition}
The scattering coefficients $\{W_{nm}\}_{m,n \in \mathbb{Z}}$ at frequency $\omega$ of the inhomogeneous scatterer
$\Omega$  with the permittivity and permeability distributions $\varepsilon, \mu$ are defined by
\beqn
    W_{nm} = W_{nm} [\varepsilon, \mu, \omega, \Omega] := \mu_0  \int_{\p \Omega} J_n (k_0 |y|) e^{-i n \theta_y} \Lambda_{\mu_0,\varepsilon_0}^{-1} ( \Lambda_{\mu_0,\varepsilon_0} - \Lambda_{\mu,\varepsilon} ) [g_m] (y) d \sigma(y). \label{eq:w_nm}
\eqn
\end{Definition}
With this definition and the derivations above, we immediately come to the following integral representation
theorem from (\ref{ingegral2}).
\begin{Theorem}
For an incident field of the form $u_0 (y) = \sum_{m\in \mathbb{Z}} a_m J_m (k_0 |y|) e^{i m \theta_y}$,
the total field $u$ (i.e., the solution of (\ref{scattering2})-(\ref{sommerfield})) has the following asymptotic representation:
\beqn
    (u - u_0)(x) = -\f{i}{4} \sum_{m,n} a_m  H^{(1)}_n (k_0 |x|) e^{i n \theta_x} W_{nm}\q
    \m{as}\q |x| \rightarrow \infty\,. \label{eq:ufromwnm}
\eqn
\end{Theorem}

\section{Representation and decay property of scattering coefficients}  \label{sec3}
In this section we would like to represent the scattering coefficients using layer potentials and
study their decay properties.  In order to do this, we first introduce the Neumann function of the Helmholtz equation and the single and double layer potentials.

Let $N_{\mu,\varepsilon} (x,y)$ be the fundamental solution to the problem (\ref{interior_pde}), i.e.,
for each fixed $z \in \Omega$, $N_{\mu,\varepsilon} (\cdot , z)$ is the solution to
\beqn
    \nabla \cdot \f{1}{\mu} \nabla N_{\mu,\varepsilon} (\cdot , z) + \omega^2 \varepsilon N_{\mu,\varepsilon} (\cdot , z) = - \delta_{z} (\cdot)   \text{ in } \Omega \, ; \quad
    \f{1}{\mu} \f{\partial}{\partial \nu} N_{\mu,\varepsilon} (\cdot , z) = 0  \text{ on } \partial \Omega\,.
    \label{Neumann}
\eqn
Let $\mathcal{N}_{\mu,\varepsilon} [g] (x) := \int_{\partial \Omega} N_{\mu,\varepsilon} (x , y) g(y) d \sigma(y)$ for $x \in \Omega$.
Then we can see that $\mathcal{N}_{\mu,\varepsilon} [g] (x)$ is the solution to (\ref{interior_pde}), and that
\beqn
    \Lambda_{\mu,\varepsilon} [g] (x) = \mathcal{N}_{\mu,\varepsilon} [g] (x) \, , \quad x \in \partial \Omega,
    \label{def_ntd}
\eqn
by noting the relation (cf.~\cite{book})
$$ \f{1}{\mu} \f{\partial}{\partial \nu} \mathcal{N}_{\mu,\varepsilon}[g]= g \quad \mbox{on } \partial \Omega\,.$$

Let $\mathcal{S}_{k_0} [\phi]$ and $\mathcal{D}_{k_0} [\phi]$ be the following
single and double layer potentials on $\p \Om$:
\beqn
    \mathcal{S}_{k_0} [\phi](x) = \int_{\partial \Omega} \Phi_{k_0}(x-y) \phi(y) d \sigma (y)\,,
    \quad x \in \mathbb{R}^2,
    \eqn
    and \beqn
    \mathcal{D}_{k_0} [\phi](x) =
     \int_{\partial \Omega} \f{\p \Phi_{k_0}}{\p \nu_y}(x-y) \phi(y) d \sigma (y), \quad x \in \mathbb{R}^2 \setminus \partial \Omega \,.
\eqn
Then the layer potentials $\mathcal{S}_{k_0}$ and $\mathcal{D}_{k_0}$ satisfy the following jump conditions:
\beqn
    \f{\p}{\p \nu} \left( \mathcal{S}_{k_0}[\phi] \right)^{\pm} = (\pm \f{1}{2} I + \mathcal{K}^*_{k_0, \Omega} )[\phi]\,,
    \quad
    \left( \mathcal{D}_{k_0}[\phi] \right)^{\pm} = (\mp \f{1}{2} I + \mathcal{K}_{k_0, \Omega} ) [\phi]\,,
    \label{jump_condition}
\eqn
where $\mathcal{K}_{k_0, \Omega}$ is the boundary integral operator defined by
\[
   \mathcal{K}_{k_0, \Omega} [\phi](x) = \int_{\partial \Omega} \f{\p \Phi_{k_0}}{\p \nu_y}(x-y) \phi(y) d \sigma (y)
\]
and $ \mathcal{K}^*_{k_0, \Omega}$ is the $L^2$ adjoint of $ \mathcal{K}_{k_0, \Omega}$ with $L^2$ being equipped with the real inner product. Note that $ \f{1}{2} I + \mathcal{K}^*_{k_0, \Omega}$ is invertible if $k_0^2$ is not a Dirichlet eigenvalue of $-\Delta$ on $\Omega$; see \cite{eigenpaper, book}. From (\ref{ingegral1}) and the transmission conditions (\ref{scattering2}), we can see that the solution $u$ to (\ref{scattering2})-(\ref{sommerfield}) can be represented as
\beqn
    u(x) =  u_0 (x) + \mu_0 \mathcal{S}_{k_0} [\phi] \text{ for } x \in \mathbb{R}^d \backslash \overline{\Omega} \, ; \quad
    u(x) =  \mathcal{N}_{\mu,\varepsilon} [\psi] \text{ for } x \in \Omega
    \label{total_u}
\eqn
for some density pair $(\phi, \psi) \in L^2(\partial \Omega) \times L^2(\partial \Omega)$ which satisfies
\[
        u_0  = \Lambda_{\mu,\varepsilon}[\psi] - \mu_0 \mathcal{S}_{k_0} [\phi] \quad \m{and} \q
        \f{1}{\mu_0} \f{\p u_0}{\p \nu}   = - ( \f{1}{2} I + \mathcal{K}^*_{k_0, \Omega} ) [\phi] + \psi \q
        \m{on} \q \partial \Omega\,.
\]
If we define
\beqn
    A := \begin{pmatrix} - \mu_0 \mathcal{S}_{k_0} & \Lambda_{\mu,\varepsilon} \\ - ( \f{1}{2} I + \mathcal{K}^*_{k_0, \Omega} ) & I  \end{pmatrix} \, , \label{operator_A}
\eqn
then we can write $(\phi,\psi)$ as the solution to the following equation
\beqn
    A \begin{pmatrix} \phi \\ \psi  \end{pmatrix} = \begin{pmatrix} u_0 \\ \f{1}{\mu_0} \f{\p u_0}{\p \nu}  \end{pmatrix} \, ,\label{potential}
\eqn
and show the following result.
\begin{Lemma}
The operator $A: L^2(\partial \Omega) \times L^2(\partial \Omega) \rightarrow L^2(\partial \Omega) \times L^2(\partial \Omega)$ is invertible.
\end{Lemma}
\begin{proof}
Let $(\phi,\psi) \in  L^2(\partial \Omega) \times L^2(\partial \Omega)$ be such that $ A \begin{pmatrix} \phi \\ \psi  \end{pmatrix} = 0$. Let $u$ be defined by
$$
u = \left\{\begin{array}{l}  \mathcal{N}_{\mu,\varepsilon} [\psi] \quad  \text{ in }  \Omega,\\
\mu_0 \mathcal{S}_{k_0} [\phi] \quad \text{ in } \mathbb{R}^d \backslash \overline{\Omega}.
\end{array}
\right.
$$
From the jump conditions
$$
\left\{
\begin{array}{l}
\mu_0 \mathcal{S}_{k_0} [\phi] = \mathcal{N}_{\mu,\varepsilon} [\psi]  \quad \text{ on } \partial \Omega, \\

\mu_0 ( \f{1}{2} I + \mathcal{K}^*_{k_0, \Omega} ) = \frac{\partial}{\partial \nu} \mathcal{N}_{\mu,\varepsilon} [\psi] = \mu_0 \psi    \quad \text{ on } \partial \Omega,
\end{array}
\right.
$$
one can see that $u$ satisfies the Helmholtz equation
(\ref{scattering1})
together with the outgoing Sommerfeld radiation condition:
\beqn
  \f{\partial}{\partial r} u  - i k_0 u  = O(|x|^{-\f{3}{2}})& \text{ as } |x| \rightarrow \infty \, .
    \label{sommerfieldo}
\eqn
Uniqueness of a solution to (\ref{scattering1}) subject to the Sommerfeld radiation condition (\ref{sommerfieldo}) shows that $u=0$ in $\mathbb{R}^d$. Then, since $k_0^2$ is not a Dirichlet eigenvalue of $-\Delta$ on $\Omega$, we have $\phi=0$, hence
$\psi=0$ as well. This shows the injectivity of $A$.

Next, since  $\f{1}{\mu_0} \Phi_{k_0}(|x-y|)$ and $N_{\mu,\varepsilon}(x,y)$ have the same singularity type (i.e., of logarithmic type) as $|x-y| \rightarrow 0$ \cite{singular} (see Appendix \ref{appendixA}) and $\mathcal{K}^*_{k_0, \Omega}$ is a compact operator on $L^2(\partial \Omega)$, it follows that $A$ is a compact perturbation of the invertible operator on $ L^2(\partial \Omega) \times L^2(\partial \Omega)$ which is given by
$$
\begin{pmatrix} - \mu_0 \mathcal{S}_{k_0} & \mu_0 \mathcal{S}_{k_0} \\ - \frac{1}{2} I & I  \end{pmatrix} \,.
$$
Therefore, Fredholm alternative holds and injectivity of $A$ shows its invertibility.
 \end{proof}

We define $(\phi_m, \psi_m)$ as the pair of solution to the above equation (\ref{potential}) corresponding to
the incident field $u_0 (y) = (u_0)_m (y) := J_m (k_0 |y|) e^{i m \theta_y} $ defined as in (\ref{incidence}), then $W_{nm}$ can be simply
expressed as
\beqn
    W_{nm} = \mu_0 \int_{\partial \Omega} J_n (k_0 |y|) e^{ - i n \theta_y} \phi_m (y) d \sigma(y) = \mu_0 \langle (u_0)_n, \phi_m \rangle_{L^2(\partial \Omega)}\, .
    \label{inner}
\eqn
Using this expression, we can derive the decay property of scattering coefficients.
Again from the fact  that the functions $\f{1}{\mu_0} \Phi_{k_0}(|x-y|)$ and $N_{\mu,\varepsilon}(x,y)$ have the same  logarithmic type singularity as $|x-y| \rightarrow 0$ \cite{singular},
we obtain from (\ref{potential}) that
\beqn
    ||\phi_m||_{L^2(\partial \Omega)} + ||\psi_m||_{L^2(\partial \Omega)} \leq C  ( ||(u_0)_m ||_{L^2(\partial \Omega)} + ||\frac{\partial}{\partial \nu} (u_0)_m ||_{L^2(\partial \Omega)} ).
\eqn
Using the asymptotic behavior of the Bessel function $J_m$  \cite{handbook},
\beqn
     J_m (t) \bigg/ \f{1}{\sqrt{2 \pi |m|}}\left(\f{e t}{2 |m|}\right)^{|m|} \rightarrow 1
     \label{decayhaha}
\eqn
as $m \rightarrow \infty$, we have
\beqnx
||(u_0)_n ||_{L^2(\partial \Omega)} \leq \f{C_1^{|n|}}{|n|^{|n|}} \q \m{and} \q
||\phi_m ||_{L^2(\partial \Omega)} \leq   \f{C_2^{|m|}}{|m|^{|m|}}
\eqnx
for some constants $C_1$ and $C_2$. Therefore, we deduce from (\ref{inner}) that
\[
    |W_{nm}| = | \mu_0 \langle (u_0)_m, \phi_m \rangle_{L^2(\partial \Omega)} |  \leq  ||(u_0)_n ||_{L^2(\partial \Omega)} ||\phi_m ||_{L^2(\partial \Omega)} \leq \f{C^{|m|+|n|}}{|m|^{|m|} |n|^{|n|}}
\]
for some constant $C$, leading to the following theorem.
\begin{Theorem}
There exists a constant $C$ depending on $(\mu,\varepsilon,\omega)$ such that
\beqn
    |W_{nm}| \leq \f{C^{|m|+|n|}}{|m|^{|m|} |n|^{|n|}} \quad \text{ for all } n, m \in \mathbb{Z} \, .
    \label{decay_conclusion}
\eqn
\end{Theorem}

\section{Far-field pattern}  \label{sec4}
In this section we shall derive the far-field pattern of the scattered field in terms of the scattering coefficients.

We consider the incident field $u_0$ as a plane wave of the form $u_0 = e^{i k_0 \xi \cdot x}$ with $\xi$ being on the unit circle.
We recall the Fourier mode $(u_0)_m (y) := J_m (k_0 |y|) e^{i m \theta_y}$ in (\ref{incidence}),
and the solution pair $(\phi_m, \psi_m)$ to (\ref{potential}) corresponding to the incident field $(u_0)_m$.  Then by the well-known Jacobi-Anger decomposition, we have the following decomposition of the plane wave in terms of $(u_0)_m$:
\beqn
    u_0 = e^{i k_0 \xi \cdot x} = \sum_{m\in \mathbb{Z}} e^{i m (\f{\pi}{2} - \theta_\xi)} J_m (k_0 |x|) e^{i m \theta_x} = \sum_{m\in \mathbb{Z}} e^{i m (\f{\pi}{2} - \theta_\xi)} (u_0)_m \, , \l{eq:u0}
\eqn
where $\xi=(\cos \theta_\xi, \sin\theta_\xi)$ and $x=|x|(\cos \theta_x, \sin \theta_x)$.

Let $(\phi, \psi)$ be the solution pair to (\ref{potential}) corresponding to the incident field $u_0 = e^{i k_0 \xi \cdot x}$,
then using (\ref{eq:u0}) and the principle of superposition we have
\beqn
    \phi = \sum_{m\in \mathbb{Z}} e^{i m (\f{\pi}{2} - \theta_\xi)} \phi_m \quad \text{ and } \quad \psi = \sum_{m\in \mathbb{Z}} e^{i m (\f{\pi}{2} - \theta_\xi)} \psi_m \, .
\eqn
It follows directly from (\ref{total_u}) that
\beqn
    u - e^{i k_0 \xi \cdot x} = \mu_0  \int_{\p \Omega} \Phi_{k_0} (x-y) \phi(y) d \sigma(y) = \mu_0 \sum_{m\in \mathbb{Z}} e^{i m (\f{\pi}{2} - \theta_\xi)}  \int_{\p \Omega} \Phi_{k_0} (x-y) \phi_m(y) d \sigma(y).
    \label{representation_plane}
\eqn
In order to derive the far-field pattern from expression (\ref{representation_plane}), we consider
the asymptotic expansion of $\Phi_{k_0} (x-y)$ as $|x| \rightarrow \infty$.
Noting the expression (\ref{fundamental}) of
$\Phi_{k_0}$ and the two approximations that
\beqn
    H^{(1)}_0 (t) = \sqrt{\f{1}{\pi t}} \left( e^{i (t - \f{\pi}{4})} + O(t^{-1}) \right) = \sqrt{\f{1}{\pi t}} e^{i (t - \f{\pi}{4})} + O(t^{-\f{3}{2}})
    \q \m{as} \q t \rightarrow \infty
\eqn
and $|x-y| = |x| - |y| \cos(\theta_x - \theta_y) + O(|x|^{-1})$ as $|x| \rightarrow \infty$,
we arrive at the following asymptotic expansion of $\Phi_{k_0} (x-y)$:
\beqn
    \Phi_{k_0} (x-y) =  e^{- i \f{\pi}{4}} \sqrt{\f{2}{\pi k_0 |x|}} e^{i k_0 \left(|x| - |y| \cos(\theta_x - \theta_y)\right) } + O(|x|^{-\f{3}{2}})
    \q \m{as} \q |x| \rightarrow \infty\,.
    \label{fundamental2}
\eqn
Substituting this into (\ref{representation_plane}) yields
\beqn
    u - e^{i k_0 \xi \cdot x} = -i e^{- i \f{\pi}{4}} \f{\mu_0 e^{i k_0 |x|} }{\sqrt{8 \pi k_0 |x|}}
    \sum_{m\in \mathbb{Z}} e^{i m (\f{\pi}{2} - \theta_\xi)}  \int_{\p \Omega}   e^{- i k_0 |y| \cos(\theta_x - \theta_y)} \phi_m(y) d \sigma (y) + O(|x|^{-\f{3}{2}}) \,,
    \label{farfield1}
\eqn
from which and the Jacobi-Anger identity
\beqn
    e^{- i k_0 |y| \cos(\theta_x - \theta_y)} = \sum_n  J_n( k_0 |y|) e^{- i n (\theta_y + \f{\pi}{2})} e^{i n \theta_x}\,
    \label{representation_plane2}
\eqn
it follows that
\[
    u (x)- e^{i k_0  \xi \cdot x} = -i e^{- i \f{\pi}{4}} \f{\mu_0 e^{i k_0 |x|} }{\sqrt{8 \pi k_0 |x|}}
    \sum_{m,n \in \mathbb{Z}}  i^{(m-n)} e^{- i m  \theta_\xi } e^{i n \theta_x} \int_{\p \Omega}  J_n( k_0 |y|) e^{- i n \theta_y }  \phi_m(y) d \sigma (y) + O(|x|^{-\f{3}{2}}) \, .
\]
Comparing this expression with the representation of $W_{nm}$ in (\ref{inner}), we infer that
\beqn
    u(x) - e^{i k_0 \xi \cdot x} = -i e^{- i \f{\pi}{4}} \f{\mu_0 e^{i k_0 |x|} }{\sqrt{8 \pi k_0 |x|}}
    \sum_{m,n \in \mathbb{Z}}  i^{(m-n)} e^{- i m  \theta_\xi } e^{i n \theta_x} W_{n m} + O(|x|^{-\f{3}{2}})  \, .
    \label{farfield_exp}
\eqn
This motivates us with the following definition of the far-field pattern.
\begin{Definition}
Consider the total field $u$ satisfying (\ref{scattering2})-(\ref{sommerfield}) with
the incident field
$u_0(x) = e^{i k_0 \xi \cdot x}$. Then
the far-field pattern $A_{\infty} [\varepsilon, \mu, \omega] (\theta_\xi, \theta_x)$ is defined by
\beqn \label{defA}
    u(x) - e^{i k_0 \xi \cdot x} = -i e^{- i \f{\pi}{4}} \f{\mu_0 e^{i k_0 |x|} }{\sqrt{8 \pi k_0 |x|}} A_{\infty} [\varepsilon, \mu, \omega] (\theta_\xi, \theta_x) + O(|x|^{-\f{3}{2}}) \q \m{as} \q |x| \rightarrow \infty\,.
\eqn
\end{Definition}

By comparing (\ref{defA}) with (\ref{farfield_exp}) we come to the following theorem.
\begin{Theorem}
Let $\theta_\xi$ and $\theta_x$ be respectively the incident and the scattered direction. Then
the far-field pattern $A_{\infty} [\varepsilon, \mu, \omega] (\theta_\xi, \theta_x)$ defined by
(\ref{defA}) can be expressed in the explicit form:
\beqn
    A_{\infty} [\varepsilon, \mu, \omega] (\theta_\xi, \theta_x) = \sum_{m,n \in \mathbb{Z}}   i^{(m-n)} e^{- i m  \theta_\xi } e^{i n \theta_x} W_{n m} [\varepsilon, \mu, \omega].
    \label{farfield_def}
\eqn
\end{Theorem}

It is easy to see that the bounds in (\ref{decay_conclusion}) ensure the converges of the above series
uniformly with respect to $\theta_\xi$ and $\theta_x$, so $A_{\infty} [\varepsilon, \mu, \omega]$ is well-defined. Moreover, one can see that reconstructing the scattering coefficients from the far-field pattern is an exponentially ill-posed  problem if the measurements of $ A_{\infty}$ are corrupted with noise.

\section{Transformation rules and properties of scattering coefficients}  \label{sec5}
In this section, we derive more properties, including some transformation rules for the scattering coefficients.
To do so, we first represent the scattering coefficients in terms of an exterior NtD map.
For any $g \in H^{-\f{1}{2}}(\p\Om)$, the action of the exterior NtD map $\Lambda_{\mu_0,\varepsilon_0}^{e}:
H^{-\f{1}{2}} (\p\Om)\to H^{\f{1}{2}} (\p\Om)$ is defined by the trace $u=\Lambda_{\mu_0,\varepsilon_0}^{e}g
\in H^{\f{1}{2}} (\p\Om)$ of the solution $u$ to the system:
\beqn
    \begin{cases}
        \f{1}{\mu_0} \Delta u + \varepsilon_0 \omega^2 u = 0 & \text{ in } \mathbb{R}^d \backslash \overline{\Omega} \, , \\
        \f{1}{\mu_0} \f{\p u}{\p \nu } = g & \text{ on }\partial \Omega \, ,\\
        \f{\partial}{\partial r} u - i k_0 u = O(|x|^{-\f{3}{2}})& \text{ as } |x| \rightarrow \infty \,.
    \end{cases}
    \label{exterior_pde}
\eqn
With the help of the exterior NtD map $\Lambda_{\mu_0,\varepsilon_0}^{e}$,
we can derive some new representation of the scattering coefficients.

\begin{Lemma}
Let $(u_0)_n$ and the scattering coefficients $W_{nm}$ be defined as in (\ref{incidence}) and
(\ref{eq:w_nm}), respectively,
and let $\Lambda_{\mu,\varepsilon}$ and $\Lambda_{\mu_0,\varepsilon_0}^{e}$  be
the interior and exterior NtD maps. Then the scattering coefficients $W_{nm}$ can be expressed as
\beqn
    W_{nm} = \langle (u_0)_n, {\cal A}_{\mu,\varepsilon}(u_0)_m \rangle_{L^2(\partial \Omega)} \quad \text{ for all } n, m \in \mathbb{Z} \, ,
    \label{wmn_bilinear}
\eqn
where the operator ${\cal A}_{\mu,\varepsilon}$ is given by
\beqn
{\cal A}_{\mu,\varepsilon} := \mu_0 \Lambda_{\mu_0,\varepsilon_0}^{-1}
\left(\Lambda_{\mu_0,\varepsilon_0} - \Lambda_{\mu,\varepsilon}\right)
\left(\Lambda_{\mu,\varepsilon} - \Lambda^{e}_{\mu_0,\varepsilon_0}\right)^{-1}
\left(\Lambda_{\mu_0,\varepsilon_0} - \Lambda^{e}_{\mu_0,\varepsilon_0}\right)
\Lambda_{\mu_0,\varepsilon_0}^{-1} .
\label{long_operator}
\eqn
\end{Lemma}
\begin{proof}
For a given incident field $u_0$, let $(\phi,\psi) \in L^2(\p \Om) \times L^2(\p \Om)$ be the density pair
that solves (\ref{potential}).  Then it follows from the jump conditions of the layer potentials in (\ref{jump_condition})
that
\beqn
\psi &=& \phi + ( - \f{1}{2} I + \mathcal{K}^*_{k_0, \Omega} ) [\phi] + \f{1}{\mu_0} \f{\p u_0}{\p \nu}
 = \phi +  \f{\p}{\p \nu} \left( \mathcal{S}_{k_0}[\phi] \right)^{-} + \f{1}{\mu_0} \f{\p u_0}{\p \nu}\,,
 \label{process1}
\eqn
\beqn
\psi &=& ( \f{1}{2} I + \mathcal{K}^*_{k_0, \Omega} ) [\phi] + \f{1}{\mu_0} \f{\p u_0}{\p \nu}
 = \f{\p}{\p \nu} \left( \mathcal{S}_{k_0}[\phi] \right)^{+}  + \f{1}{\mu_0} \f{\p u_0}{\p \nu} \, .
\label{process2}
\eqn
By directly applying the interior and exterior NtD operators to (\ref{process1}) and (\ref{process2}), we obtain
\[
    \begin{cases}
    \Lambda_{\mu_0,\varepsilon_0}^{e} [\psi] & = \mu_0 \mathcal{S}_{k_0}[\phi] + \f{1}{\mu_0} \Lambda_{\mu_0,\varepsilon_0}^{e} \left[\f{\p u_0}{\p \nu}  \right] \, , \\
    \Lambda_{\mu_0,\varepsilon_0}[\psi] & = \Lambda_{\mu_0,\varepsilon_0}[\phi] +  \mu_0  \mathcal{S}_{k_0}[\phi] + u_0  \, , \\
    \Lambda_{\mu,\varepsilon}[\psi] & = u_0 + \mu_0 \mathcal{S}_{k_0} [\phi] \, ,
    \end{cases}
\]
which combines to give
\beqn
    \begin{cases}
    (\Lambda_{\mu,\varepsilon} - \Lambda_{\mu_0,\varepsilon_0}^{e} ) [\psi] & = \f{1}{\mu_0} ( \Lambda_{\mu_0,\varepsilon_0}- \Lambda_{\mu_0,\varepsilon_0}^{e}) \left[\f{\p u_0}{\p \nu} \right] =  ( \Lambda_{\mu_0,\varepsilon_0}- \Lambda_{\mu_0,\varepsilon_0}^{e}) \Lambda_{\mu_0,\varepsilon_0}^{-1} \left[ u_0 \right] \, , \\
    (\Lambda_{\mu_0,\varepsilon_0} - \Lambda_{\mu,\varepsilon})[\psi]& = \Lambda_{\mu_0,\varepsilon_0}[\phi] \, .
    \end{cases}
    \label{process3}
\eqn
Substituting the first equation in (\ref{process3}) into the second, we readily get
\[
    \phi = \Lambda_{\mu_0,\varepsilon_0}^{-1}(\Lambda_{\mu_0,\varepsilon_0} - \Lambda_{\mu,\varepsilon})[\psi] = \Lambda_{\mu_0,\varepsilon_0}^{-1}(\Lambda_{\mu_0,\varepsilon_0} - \Lambda_{\mu,\varepsilon})(\Lambda_{\mu,\varepsilon} - \Lambda_{\mu_0,\varepsilon_0}^{e} )^{-1} ( \Lambda_{\mu_0,\varepsilon_0}- \Lambda_{\mu_0,\varepsilon_0}^{e}) \Lambda_{\mu_0,\varepsilon_0}^{-1} \left[ u_0 \right].
\]
In particular, if $(\phi_m,\psi_m) \in L^2(\p \Om) \times L^2(\p \Om)$ be the density pair  that satisfies (\ref{potential}) corresponding to
the incident field $u_0 (y) = (u_0)_m (y) := J_m (k_0 |y|) e^{i m \theta_y} $ as in (\ref{incidence}), then $\phi_m$ satisfies
\beqn
    \phi_m = \Lambda_{\mu_0,\varepsilon_0}^{-1}(\Lambda_{\mu_0,\varepsilon_0} - \Lambda_{\mu,\varepsilon})(\Lambda_{\mu,\varepsilon} - \Lambda_{\mu_0,\varepsilon_0}^{e} )^{-1} ( \Lambda_{\mu_0,\varepsilon_0}- \Lambda_{\mu_0,\varepsilon_0}^{e}) \Lambda_{\mu_0,\varepsilon_0}^{-1} \left[ (u_0)_m \right] = \f{1}{\mu_0} {\cal A}_{\mu,\varepsilon} (u_0)_m\,.
    \label{potential_property}
\eqn
Substituting (\ref{potential_property}) into (\ref{inner}), we conclude that
\[
    W_{nm} = \mu_0 \langle (u_0)_n, \phi_m \rangle_{L^2(\partial \Omega)} = \langle (u_0)_n, {\cal A}_{\mu,\varepsilon}(u_0)_m \rangle_{L^2(\partial \Omega)}\,.
\]
\end{proof}

With the representations (\ref{inner}) and (\ref{wmn_bilinear}), we can derive
some special transformation rules for the scattering coefficients.
\begin{Corollary}
\label{coll_four}
The scattering coefficients $\{W_{nm}\}_{n, m \in \mathbb{Z}}$ in (\ref{eq:w_nm}) meet the following transformation rules:
\begin{enumerate}
\item
$  W_{nm} [\varepsilon, \mu, \omega, \Omega] = \overline{ W_{mn} [\varepsilon, \mu, \omega, \Omega]} $;
\item
$ W_{nm} [\varepsilon, \mu, \omega, e^{i \theta} \Omega]  = e^{i (m-n)\theta} W_{nm} [\varepsilon, \mu, \omega, \Omega] $ ~for all $\theta \in [0, 2 \pi]$;
\item
$ W_{nm} [\varepsilon, \mu, \omega, s \Omega] = W_{nm} [\varepsilon, \mu, s \omega, \Omega] $ ~for all $s >0 $;
\item
$ W_{nm} [\varepsilon, \mu, \omega, \Omega + z ]= \sum_{l,l \in \mathbb{Z}} \overline{(u_0)_{p}(z)}(u_0)_{l}(z) W_{n-p,m-l} [\varepsilon, \mu, \omega, \Omega ]$ ~for all $z \in \mathbb{R}^2$,
\end{enumerate}
where we identify the spaces before and after translation, rotation and scaling
by the natural isomorphism, e.g., $H^s (\partial \Omega) \cong H^s (e^{i \theta} \partial \Omega)$.
\end{Corollary}

\begin{proof}
We start with the first result in Corollary\,\ref{coll_four}. From representation (\ref{wmn_bilinear}) of $W_{nm}$, it suffices to show that the operator ${\cal A}_{\mu,\varepsilon}$ defined in (\ref{long_operator}) is self-adjoint. To do this, we utilize the following identity
for any operators $A$, $B$, $C$ such that $A-C$ and $B-C$ are invertible:
\beqn
(A-C)^{-1} - (B-C)^{-1} = (A-C)^{-1} (B-A) (B-C)^{-1} = (B-C)^{-1} (B-A) (A-C)^{-1} \, .
\label{operator_identity}
\eqn
Using this we can write
\beqn
    (\Lambda_{\mu_0,\varepsilon_0}-\Lambda_{\mu_0,\varepsilon_0}^e)^{-1} - (\Lambda_{\mu,\varepsilon}-\Lambda_{\mu_0,\varepsilon_0}^e)^{-1} = (\Lambda_{\mu,\varepsilon}-\Lambda_{\mu_0,\varepsilon_0}^e)^{-1} (\Lambda_{\mu,\varepsilon}-\Lambda_{\mu_0,\varepsilon_0}) (\Lambda_{\mu_0,\varepsilon_0}-\Lambda_{\mu_0,\varepsilon_0}^e)^{-1}.
    \label{operator_identity2}
\eqn
Substituting (\ref{operator_identity2}) into (\ref{long_operator}), we get
\beqnx
&& {\cal A}_{\mu,\varepsilon} \\
&=& \mu_0 \Lambda_{\mu_0,\varepsilon_0}^{-1}
\left(\Lambda_{\mu_0,\varepsilon_0} - \Lambda_{\mu,\varepsilon}\right)
\left(\Lambda_{\mu,\varepsilon} - \Lambda^{e}_{\mu_0,\varepsilon_0}\right)^{-1}
\left(\Lambda_{\mu_0,\varepsilon_0} - \Lambda^{e}_{\mu_0,\varepsilon_0}\right)
\Lambda_{\mu_0,\varepsilon_0}^{-1} \\
&=& \mu_0 \Lambda_{\mu_0,\varepsilon_0}^{-1} (\Lambda_{\mu_0,\varepsilon_0} - \Lambda_{\mu,\varepsilon}) \Lambda_{\mu_0,\varepsilon_0}^{-1} +\mu_0 \Lambda_{\mu_0,\varepsilon_0}^{-1}  ( \Lambda_{\mu,\varepsilon} - \Lambda_{\mu_0,\varepsilon_0}) (\Lambda_{\mu,\varepsilon}-\Lambda_{\mu_0,\varepsilon_0}^e)^{-1}  (\Lambda_{\mu,\varepsilon}-\Lambda_{\mu_0,\varepsilon_0}) \Lambda_{\mu_0,\varepsilon_0}^{-1}\,.
\eqnx
Now the self-adjointness of ${\cal A}_{\mu,\varepsilon}$ is a consequence of the self-adjointness of $\Lambda_{\mu_0,\varepsilon_0}$,
$\Lambda_{\mu,\varepsilon}$ and $\Lambda_{\mu_0,\varepsilon_0}^e$.

%
To see the second result in Corollary\,\ref{coll_four}, we consider the change of coordinates from $(|y|, \theta_y)$ to $(|\widetilde{y}|, \widetilde{\theta_y})$, with $\widetilde{\theta_y} + \theta = \theta_y$ and $|\widetilde{y}| = |y|$.
It follows from definition (\ref{incidence}) that $(u_0)_m (y)= J_m (k_0|y|) e^{i m \theta_y} = J_m (k_0|\widetilde{y}|) e^{i m (\widetilde{\theta_y} + \theta)} = (u_0)_m (\widetilde{y}) e^{i m \theta}$.
Let $\widetilde{u_m}(\widetilde{y})$ be the solution to
\beqn
\begin{cases}
    \nabla_{\widetilde{y}} \cdot \left( \f{1}{\mu(\widetilde{y})} \nabla_{\widetilde{y}} u(\widetilde{y}) \right) + \omega^2 \varepsilon(\widetilde{y}) u(\widetilde{y}) = 0  & \text{ in } \mathbb{R}^2 \, , \\
    \f{\partial}{\partial r} (u - u_0)(\widetilde{y}) - i k_0 (u - u_0)(\widetilde{y}) = O(|\widetilde{y}|^{-\f{3}{2}})& \text{ as } |x| \rightarrow \infty \, .
\end{cases}
\eqn
with the incident field $u_0(\widetilde{y}) = (u_0)_m (\widetilde{y})$, and let $u_m(y)$ be
the solution to (\ref{scattering1})-(\ref{sommerfield}) with the incident field $u_0(y)=(u_0)_m (y) = (u_0)_m (\widetilde{y}) e^{i m \theta}$.  Then we can see that $u_m(y)$ is actually $u_m(y) = \widetilde{u_m}(\widetilde{y}) e^{i m \theta}$.  Therefore we observe that the density pair $(\psi_m,\phi_m)$ satisfying (\ref{potential}) with incident field $u_0 (y) = (u_0)_m (y)$ and electromagnetic parameters $\mu(y)$ and $\varepsilon(y)$ actually has the form $ (\psi_m(y),  \phi_m(y))= (\widetilde{\phi}_m (\widetilde{y}),
\widetilde{\phi}_m (\widetilde{y}) )e^{i m \theta} $, where $(\widetilde{\psi}_m,\widetilde{\phi}_m)$ satisfies (\ref{potential}) with incident field $(u_0)_m (\widetilde{y})$ and parameters $\mu(\widetilde{y})$, $\varepsilon(\widetilde{y})$.  Hence we derive from (\ref{inner}) that
\beqnx
    W_{nm}[\varepsilon, \mu, \omega, e^{i \theta} \Omega]  &=& \mu_0 \int_{e^{i\theta} \partial \Omega} J_n(k_0|y|) e^{-i n \theta_y} \phi_m d \sigma(y) \\
    &=& \mu_0 \int_{e^{i\theta} \partial \Omega} J_n(k_0|\widetilde{y}|) e^{-i n (\widetilde{\theta_y} +\theta)} \widetilde{\phi_m}(y)(\widetilde{\theta_y}) e^{i m \theta} d \sigma(y) \\
    &=& e^{i (m-n) \theta} \mu_0 \int_{e^{i \theta} \partial \Omega} J_n(k_0|\tilde{y}|) e^{-i n \widetilde{\theta_y}} \widetilde{\phi_m}(\widetilde{y})  d \sigma(\widetilde{y}) \\
    &=& e^{i (m-n)\theta} W_{nm} [\varepsilon, \mu, \omega, \Omega].
\eqnx
This proves the second result in Corollary\,\ref{coll_four}.

Next for the third result in Corollary\,\ref{coll_four},
we consider the change of coordinates from $(|y|, \theta_y)$ to $(|\widetilde{y}|, \widetilde{\theta_y})$, with
$\widetilde{\theta_y} = \theta_y$ and $s|\widetilde{y}| = |y|$.
We know from (\ref{incidence}) that
$(u_0)_m (y)=  J_m (k_0|y|) e^{i m \theta_y} = J_m (s|\widetilde{y}|) e^{i m \widetilde{\theta_y})}$.
Let $\widetilde{u_m}(\widetilde{y})$ be the solution to the following system
\beqn
\begin{cases}
    \nabla_{\widetilde{y}} \cdot \left( \f{1}{\mu(\widetilde{y})} \nabla_{\widetilde{y}} u(\widetilde{y}) \right) + (s\omega)^2 \varepsilon(\widetilde{y}) u(\widetilde{y}) = 0  & \text{ in } \mathbb{R}^2 \, , \\
    \f{\partial}{\partial r} (u - u_0)(\widetilde{y}) - i k_0 (u - u_0)(\widetilde{y}) = O(|\widetilde{y}|^{-\f{3}{2}})& \text{ as } |x| \rightarrow \infty \,
\end{cases}
\eqn
with the incident field $u_0(\widetilde{y}) = (u_0)_m (\widetilde{y})$, then it is easy to see
that the solution $u_m(y)$ to the system (\ref{scattering1})-(\ref{sommerfield})
with the incident field $u_0(y)=(u_0)_m (y) = (u_0)_m (s \widetilde{y})$ takes the form $u_m(y) = \widetilde{u_m}( s \widetilde{y})$. With this, we observe that the density pair $(\psi_m,\phi_m)$ satisfying (\ref{potential}) with incident field $u_0 (y) = (u_0)_m (y)$ and parameters $\mu(y)$ and $\varepsilon(y)$ is given by
$(\psi_m(y),  \phi_m(y) )=  (\widetilde{\psi}_m ( s \widetilde{y}),  \widetilde{\phi}_m ( s \widetilde{y}))/s$
with $(\widetilde{\psi}_m,\widetilde{\phi}_m)$ satisfying (\ref{potential}) with incident field $(u_0)_m (\widetilde{y})$ and parameters $\mu(\widetilde{y})$, $\varepsilon(\widetilde{y})$.  This comes from the fact that
$\f{\p u_m^{-}}{\p \nu_y} =  \f{\p \widetilde{u_m}^- }{\p \nu_{\widetilde{y}}} ( s \widetilde{y}) /s=
 \widetilde{\psi}_m ( s \widetilde{y})/s$, by comparing (\ref{process3}) with (\ref{eq:w_nm}) and (\ref{inner}).
Now the desired third result in Corollary\,\ref{coll_four} follows from the straightforward derivations:
\beqnx
    W_{nm} [\varepsilon, \mu, \omega, s \Omega]  &=& \mu_0 \int_{s \partial \Omega} J_n(k_0|y|) e^{-i n \theta_y} \phi_m(y) d \sigma(y) \\
    &=& \mu_0 \f{1}{s} \int_{s \partial \Omega} J_n(k_0 s|\widetilde{y}|) e^{-i n \theta_y} \widetilde{\phi}_m(s |\widetilde{y}|) d \sigma(y) \\
    &=& \mu_0 \int_{\partial \Omega} J_n(k_0 s|\widetilde{y}|) e^{-i n \theta_y}  \widetilde{\phi}_m(s |\widetilde{y}|) d \sigma (\widetilde{y}) \\
    &=& W_{nm} [\varepsilon, \mu, s \omega, \Omega].
\eqnx

Finally we come to derive the last relation in Corollary\,\ref{coll_four}. To do so,
we consider the change of coordinates from $(|y|, \theta_y)$ to $(|\widetilde{y}|, \widetilde{\theta_y})$ that has
point $z$ as the origin.  Then the definition of $(u_0)_m$ in (\ref{incidence}) and
the Graf's addition formula (\ref{graf}) allow us to write
$$
(u_0)_m = J_m (k_0|y|) e^{i m \theta_y} = \sum_{a \in \mathbb{Z}} J_a (k_0|z|) e^{i m \theta_z} J_{m-a}(k_0|\widetilde{y}|) e^{i (m-a) \widetilde{\theta_y}}.
$$
By the linearity of operator $A$ in (\ref{operator_A}), the density pair $(\psi_m,\phi_m)$
 satisfying (\ref{potential}) with the incident field $u_0 (y) = (u_0)_m (y)$ can be expressed
 in the form $(\psi_m,  \phi_m)= \sum_{a \in \mathbb{Z}}  J_a (k_0|z|) e^{i m \theta_z} (\widetilde{\psi}_{m-a} (\widetilde{y}),
 \widetilde{\phi}_{m-a} (\widetilde{y}))$,
 where $(\widetilde{\psi}_m,\widetilde{\phi}_m)$ satisfies (\ref{potential}) with the incident field $(u_0)_m (\widetilde{y})$.  With these preparations,
 the last result in Corollary\,\ref{coll_four} follows readily from the following derivations:
 \beqnx
    &&W_{nm} [\varepsilon, \mu, \omega, \Omega + z ] \\
    &=& \mu_0 \int_{\partial \Omega + z} J_n(k_0|y|) e^{-i n \theta_y} \phi_m (y) d \sigma(y) \\
    &=& \mu_0 \sum_{b\in \mathbb{Z}} J_a (k_0|z|) e^{i m \theta_z} \int_{\partial \Omega + z} J_{n-b}(k_0|\widetilde{y}|) e^{-i (n-b) \widetilde{\theta_y}} \phi_m (y) d \sigma(y) \\
    &=& \mu_0 \sum_{a,b\in \mathbb{Z}} J_b (k_0|z|) e^{- i m \theta_z}  J_a (k_0|z|) e^{i m \theta_z} \int_{\partial \Omega}  J_{n-b}(k_0|\widetilde{y}|) e^{- i (n-b) \widetilde{\theta_y}} \widetilde{\phi}_{m-a} (\widetilde{y},\widetilde{\theta_y}) d \sigma(\widetilde{y}) \\
    &=& \sum_{a,b\in \mathbb{Z}} \overline{(u_0)_{b}(z)}(u_0)_{a}(z) W_{n-b,m-a} [\varepsilon, \mu, \omega, \Omega ]\,.
\eqnx
\end{proof}

We end this section with one more representation of $W_{nm}$.
\begin{Lemma}
Let $(u_0)_m$ be defined as in (\ref{incidence}) and $u_m$ be the solution to (\ref{scattering1})-(\ref{sommerfield})
with the incident field $(u_0)_m$. Then the scattering coefficients in (\ref{eq:w_nm}) admits
the following  representation for any $n, m \in \mathbb{Z}$:
\beqnx
    W_{nm} = \omega^2 \mu_0 \int_{\Omega} \left(\varepsilon_0(y)-\varepsilon(y)\right) \overline{(u_0)_n}(y) u_m(y) d \sigma(y) + \mu_0 \int_{\Omega} \left(\f{1}{\mu(y)}-\f{1}{\mu_0(y)}\right) \overline{\nabla (u_0)_n}(y) \nabla u_m(y) d \sigma(y) \, . \label{new_representation_Wnm}
\eqnx
\end{Lemma}
\begin{proof}
Let $(\psi_m,\phi_m)$ be the density pair $(\psi_m,\phi_m)$ that satisfies (\ref{potential}) with
the incident field $u_0 (y) = (u_0)_m (y)$. Then it follows directly from (\ref{inner}), (\ref{jump_condition}) and
(\ref{potential}) that
\beqnx
 W_{nm}& = & \mu_0 \int_{\partial \Omega} \overline{(u_0)_n} (y) \phi_m (y) d \sigma(y) \\
 & = & \mu_0 \int_{\partial \Omega} \overline{(u_0)_n} (y) \left[\f{\p \left(\mathcal{S}_{k_0}[\phi_m]\right)^+}{\p \nu} (y) - \f{\p \left(\mathcal{S}_{k_0}[\phi_m]\right)^-}{\p \nu} (y) \right] d \sigma(y) \\
  & = & \mu_0 \int_{\partial \Omega} \overline{(u_0)_n} (y)  \left( \psi_m(y) - \f{1}{\mu_0} \f{\p (u_0)_m}{\p \nu} (y) \right) d \sigma(y) -  \mu_0\int_{\partial \Omega} \overline{(u_0)_n} (y) \f{\p \left(\mathcal{S}_{k_0}[\phi_m]\right)^-}{\p \nu} (y) d \sigma(y)\,.
\eqnx
Using  Green's identity and (\ref{potential}), we can further derive
\beqnx
W_{nm}  & = & \mu_0 \int_{\partial \Omega} \overline{(u_0)_n} (y)  \left( \psi_m(y) - \f{1}{\mu_0} \f{\p (u_0)_m}{\p \nu} (y) \right) d \sigma(y) -  \mu_0 \int_{\partial \Omega} \f{\p \overline{(u_0)_n}}{\p \nu}(y) \mathcal{S}_{k_0}[\phi_m](y) d \sigma(y)\\
  & = & \mu_0 \int_{\partial \Omega} \overline{(u_0)_n} (y)  \left( \psi_m- \f{1}{\mu_0} \f{\p (u_0)_m}{\p \nu}  \right) d \sigma(y) - \int_{\partial \Omega} \f{\p \overline{(u_0)_n}}{\p \nu}(y) \left(\Lambda_{\mu,\varepsilon}[\psi_m] - (u_0)_m\right)  d \sigma(y) \\
  & = & \mu_0 \int_{\partial \Omega} \overline{(u_0)_n} (y)   \psi_m(y)  d \sigma(y) - \int_{\partial \Omega} \f{\p \overline{(u_0)_n}}{\p \nu}(y) \Lambda_{\mu,\varepsilon}[\psi_m](y)  d \sigma(y)\,.
\eqnx
Now the desired representation of $W_{nm}$ follows from (\ref{total_u}) and (\ref{def_ntd}),
the comparison of (\ref{process3}) with (\ref{eq:w_nm}) and (\ref{inner}), and the Green's identity:
\beqnx
W_{nm}  & = &\mu_0 \int_{\partial \Omega} \overline{(u_0)_n} (y)   \psi_m(y)  d \sigma(y) - \int_{\partial \Omega} \f{\p \overline{(u_0)_n}}{\p \nu}(y) u_m(y)  d \sigma(y) \\
  & = & \mu_0 \int_{\partial \Omega} \overline{(u_0)_n} (y)  \f{1}{\mu} \f{\p u_m^{-}}{\p \nu}  d \sigma(y) - \int_{\partial \Omega} \f{\p \overline{(u_0)_n}}{\p \nu}(y) u_m(y)  d \sigma(y) \\
 & = & \omega^2 \mu_0 \int_{\Omega} \left(\varepsilon_0-\varepsilon\right) \overline{(u_0)_n}(y) u_m(y) d \sigma(y) + \mu_0 \int_{\Omega} \left(\f{1}{\mu}-\f{1}{\mu_0}\right) \overline{\nabla (u_0)_n}(y) \nabla u_m(y) d \sigma(y).
\eqnx
\end{proof}

\section{Sensitivity analysis}  \label{sec6}
In this section, we shall investigate the sensitivity of the scattering coefficients with respect to the changes in the permittivity and permeability distributions.  This will provide us with  perturbation formulas for evaluating the gradients that are needed in numerical minimization algorithms for reconstructing the permittivity and permeability distributions.

We study a perturbation of $W_{nm}$ for $n,m \in \mathbb{Z}$ with respect to a change
of $(\mu, \varepsilon)$. More specifically, we consider the difference $W_{nm}^{\delta} - W_{nm}$
between
\beqn
        W_{nm}^{\delta} : = W_{nm} \left[ \, \varepsilon^\delta , \mu^{\delta}, \omega, \Omega \, \right] \q \m{and}\q
        W_{nm} := W_{nm}[ \, \varepsilon, \mu, \omega, \Omega]
    \label{perturbed}
\eqn
in terms of the differences $\varepsilon^\delta - \varepsilon$ and ${1}/{\mu^{\delta}} - {1}/{\mu}$, where $(\mu,\varepsilon)$ and
$(\mu^\delta,\varepsilon^\delta)$ are two different sets of electromagnetic parameters.
In the subsequent analysis, we shall often write
\beqn
\widehat{\varepsilon} := \left\{ || \varepsilon^\delta - \varepsilon ||^2_{L^\infty(\Omega)} + \bigg|\bigg|\f{1}{\mu^{\delta}} - \f{1}{\mu}\bigg|\bigg|^2_{L^\infty(\Omega)}\right\}^{1/2}.
\label{orderepsilon}
\eqn
We first note that if $\widehat{\varepsilon}$ 
is small enough, then the NtD map $ \Lambda_{\mu^\delta,\varepsilon^\delta}$ is well defined provided that $\Lambda_{\mu,\varepsilon}$ is well defined. This follows from the theory of collectively compact operators; see \cite{anselone, book}.

Next we show
the following expression for the difference $W_{nm}^{\delta} - W_{nm}$.
\begin{Lemma}
For all $n,m \in \mathbb{Z}$,
the difference $W_{nm}^{\delta} - W_{nm}$ can be represented in terms of
the interior and exterior NtD maps $\Lambda_{\mu,\varepsilon}$ and $\Lambda_{\mu_0,\varepsilon_0}^{e}$
as follows:
\beqn
    W_{nm}^{\delta} - W_{nm}
    & = &\mu_0
    \int_{\partial \Omega} \overline{\psi_n} (y)
    \left(\Lambda_{\mu,\varepsilon} - \Lambda_{\mu^\delta,\varepsilon^\delta}\right)
    [\psi_m^\delta] (y) d \sigma(y) \, ,
    \label{B_bilinear_useful}
\eqn
where $\psi_n$ and $\psi_m^\delta$ are given by
\beqn
    \psi_n &=&
    \left(\Lambda_{\mu,\varepsilon} - \Lambda^{e}_{\mu_0,\varepsilon_0}\right)^{-1}
    \left(\Lambda_{\mu_0,\varepsilon_0} - \Lambda^e_{\mu_0,\varepsilon_0}\right)
    \Lambda_{\mu_0,\varepsilon_0}^{-1} (u_0)_n\, , \label{recall1} \\
    \psi_m^\delta &=& \left(\Lambda_{\mu^\delta,\varepsilon^\delta} - \Lambda^{e}_{\mu_0,\varepsilon_0}\right)^{-1}
    \left(\Lambda_{\mu_0,\varepsilon_0} - \Lambda^e_{\mu_0,\varepsilon_0}\right)
    \Lambda_{\mu_0,\varepsilon_0}^{-1}
    (u_0)_m \, .
    \label{recall2}
\eqn
\end{Lemma}
\begin{proof}
Using the identity (\ref{operator_identity}) we can write
\beqn
(\Lambda_{\mu^\delta,\varepsilon^\delta}-\Lambda^{e}_{\mu_0,\varepsilon_0})^{-1} - (\Lambda_{\mu,\varepsilon}-\Lambda^{e}_{\mu_0,\varepsilon_0})^{-1} = (\Lambda_{\mu,\varepsilon}-\Lambda^{e}_{\mu_0,\varepsilon_0})^{-1} (\Lambda_{\mu,\varepsilon}-\Lambda_{\mu^\delta,\varepsilon^\delta}) (\Lambda_{\mu^\delta,\varepsilon^\delta}-\Lambda^{e}_{\mu_0,\varepsilon_0})^{-1} \,,
\label{identity_00}
\eqn
which enables us to derive
\beqn
&&\left(\Lambda_{\mu_0,\varepsilon_0} - \Lambda_{\mu^\delta,\varepsilon^\delta}\right)
\left(\Lambda_{\mu^\delta,\varepsilon^\delta} - \Lambda^{e}_{\mu_0,\varepsilon_0}\right)^{-1}
-
\left(\Lambda_{\mu_0,\varepsilon_0} - \Lambda_{\mu,\varepsilon}\right)
\left(\Lambda_{\mu,\varepsilon} - \Lambda^{e}_{\mu_0,\varepsilon_0}\right)^{-1} \notag \\
&=&
\left(\Lambda_{\mu,\varepsilon} - \Lambda_{\mu^\delta,\varepsilon^\delta}\right)
\left(\Lambda_{\mu^\delta,\varepsilon^\delta} - \Lambda^{e}_{\mu_0,\varepsilon_0}\right)^{-1}
+
\left(\Lambda_{\mu_0,\varepsilon_0} - \Lambda_{\mu,\varepsilon}\right)
\left[ \left(\Lambda_{\mu^\delta,\varepsilon^\delta} - \Lambda^{e}_{\mu_0,\varepsilon_0}\right)^{-1}
-\left(\Lambda_{\mu,\varepsilon} - \Lambda^{e}_{\mu_0,\varepsilon_0}\right)^{-1} \right] \notag \\
&=&
\left [ I +
\left(\Lambda_{\mu_0,\varepsilon_0} - \Lambda_{\mu,\varepsilon}\right)
(\Lambda_{\mu,\varepsilon}-\Lambda^{e}_{\mu_0,\varepsilon_0})^{-1} \right]
(\Lambda_{\mu,\varepsilon}-\Lambda_{\mu^\delta,\varepsilon^\delta}) (\Lambda_{\mu^\delta,\varepsilon^\delta}-\Lambda^{e}_{\mu_0,\varepsilon_0})^{-1} \notag \\
&=&
\left(\Lambda_{\mu_0,\varepsilon_0}-\Lambda^{e}_{\mu_0,\varepsilon_0}\right)
(\Lambda_{\mu,\varepsilon}-\Lambda^{e}_{\mu_0,\varepsilon_0})^{-1}
(\Lambda_{\mu,\varepsilon}-\Lambda_{\mu^\delta,\varepsilon^\delta}) (\Lambda_{\mu^\delta,\varepsilon^\delta}-\Lambda^{e}_{\mu_0,\varepsilon_0})^{-1} \, .
\label{identity_01}
\eqn
It follows directly from (\ref{identity_01}) and definition (\ref{long_operator}) for the operators ${\cal A}_{\mu,\varepsilon}$ and
${\cal A}_{\mu^\delta,\varepsilon^\delta}$ that
{\footnotesize
\beqnx
&&{\cal A}_{\mu^\delta,\varepsilon^\delta} - {\cal A}_{\mu,\varepsilon} \\
&=&\mu_0 \Lambda_{\mu_0,\varepsilon_0}^{-1}
\left\{
\left(\Lambda_{\mu_0,\varepsilon_0} - \Lambda_{\mu^\delta,\varepsilon^\delta}\right)
\left(\Lambda_{\mu^\delta,\varepsilon^\delta} - \Lambda^{e}_{\mu_0,\varepsilon_0}\right)^{-1}
-
\left(\Lambda_{\mu_0,\varepsilon_0} - \Lambda_{\mu,\varepsilon}\right)
\left(\Lambda_{\mu,\varepsilon} - \Lambda^{e}_{\mu_0,\varepsilon_0}\right)^{-1}
\right\}
\left(\Lambda_{\mu_0,\varepsilon_0} - \Lambda^{e}_{\mu_0,\varepsilon_0}\right)
\Lambda_{\mu_0,\varepsilon_0}^{-1} \\
&=&\mu_0 \Lambda_{\mu_0,\varepsilon_0}^{-1}
\left(\Lambda_{\mu_0,\varepsilon_0}-\Lambda^{e}_{\mu_0,\varepsilon_0}\right)
(\Lambda_{\mu,\varepsilon}-\Lambda^{e}_{\mu_0,\varepsilon_0})^{-1}
(\Lambda_{\mu,\varepsilon}-\Lambda_{\mu^\delta,\varepsilon^\delta}) (\Lambda_{\mu^\delta,\varepsilon^\delta}-\Lambda^{e}_{\mu_0,\varepsilon_0})^{-1}
\left(\Lambda_{\mu_0,\varepsilon_0} - \Lambda^{e}_{\mu_0,\varepsilon_0}\right)
\Lambda_{\mu_0,\varepsilon_0}^{-1} \, .
\eqnx
}Now identity (\ref{B_bilinear_useful}) is a consequence of the above relation and
the representation (\ref{wmn_bilinear}) for $W_{nm}$ and $W_{nm}^{\delta}$,
\[
    W_{nm}^{\delta} - W_{nm} = \langle (u_0)_n, \left({\cal A}_{\mu^\delta,\varepsilon^\delta} - {\cal A}_{\mu,\varepsilon} \right) (u_0)_m \rangle_{L^2(\partial \Omega)} = \langle \psi_n, \left(\Lambda_{\mu,\varepsilon} - \Lambda_{\mu^\delta,\varepsilon^\delta}\right)
    [\psi_m^\delta] \rangle_{L^2(\partial \Omega)} \, ,
\]
where $\langle , \rangle$ denotes the complex inner product on $L^2(\partial \Omega)$.
\end{proof}

The following identity will be useful for the subsequent analysis.
\begin{Lemma}
\label{quadratic}
For the solutions
$u_i$ ($i = 1,2$) to the two systems
\beqn
    \nabla \cdot \left( \f{1}{\mu_i} \nabla u_i \right)+ \omega^2 \varepsilon_i u_i = 0  \text{ in } \Omega \,;  \quad
    \f{1}{\mu_i} \f{\partial u_i}{\partial \nu} = g  \text{ on } \partial \Omega\,,
    \label{systems_purb}
\eqn
the following identity holds
\beqn
    & & \int_{\partial \Omega} \overline{g}
    \left(\Lambda_{\mu_2,\varepsilon_2} - \Lambda_{\mu_1,\varepsilon_1}\right)
    [g] d \sigma \notag \\
    & = & \f{1}{2} \int_{\Omega} \left( \f{1}{\mu_1} - \f{1}{\mu_2} \right) \left( -| \nabla (u_1 - u_2) |^2 + | \nabla u_1 |^2 + | \nabla u_2 |^2\right) \notag \\
    &   & - \f{1}{2} \omega^2 \int_{\Omega} \left( \varepsilon_1 - \varepsilon_2 \right) \left( -| u_1 - u_2 |^2 + | u_1 |^2 + | u_2 |^2\right) dx \, .
    \label{quadraticpolarization}
\eqn
\end{Lemma}
\begin{proof}
It follows easily from (\ref{systems_purb}) and integration by parts that
\beqn
\int_{\partial \Omega} \overline{g} (\Lambda_{\mu_1,\varepsilon_1})[g] \, d \sigma &=& \int_{\Omega} \left( \f{1}{\mu_1} |\nabla u_1|^2 - \omega^2 \varepsilon_1 |u_1|^2 \right) dx \, , \label{101}\\
\int_{\partial \Omega} \overline{g} (\Lambda_{\mu_2,\varepsilon_2})[g] \, d \sigma &=& \int_{\Omega} \left( \f{1}{\mu_1} \overline{\nabla u_1} \cdot \nabla u_2  - \omega^2 \varepsilon_1 \overline{u_1} u_2 \right) dx \, , \label{102}\\
\int_{\partial \Omega} \overline{g} (\Lambda_{\mu_2,\varepsilon_2})[g] \, d \sigma &=& \int_{\Omega} \left( \f{1}{\mu_2} |\nabla u_2|^2 - \omega^2 \varepsilon_2 |u_2|^2 \right) dx \, , \label{103}\\
\int_{\partial \Omega} \f{1}{\mu_1} \overline{\f{\partial u_2}{\partial \nu}} u_2 \, d \sigma &=& \int_{\Omega} \left( \f{1}{\mu_1} |\nabla u_2|^2 - \omega^2 \varepsilon_1 |u_2|^2 \right) dx \, . \label{104}
\eqn
Combining (\ref{101})-(\ref{104}) yields
\beqnx
    & &
    \int_{\Omega}  \f{1}{\mu_1}  | \nabla(u_2 - u_1) |^2 dx
    -  \omega^2 \int_{\Omega} \varepsilon_1  | u_2 - u_1 |^2 dx
    +  \int_{\Omega} \left( \f{1}{\mu_2} - \f{1}{\mu_1} \right) | \nabla u_2 |^2 dx
    -  \omega^2 \int_{\Omega} \left( \varepsilon_2 - \varepsilon_1 \right)  | u_2 |^2 dx \\
    &=&
    \int_{\partial \Omega} \overline{g} (\Lambda_{\mu_1,\varepsilon_1})[g] \, d \sigma
    - 2 \int_{\partial \Omega} \overline{g} (\Lambda_{\mu_2,\varepsilon_2})[g] \, d \sigma
    + \int_{\partial \Omega} \f{1}{\mu_1} \overline{\f{\partial u_2}{\partial \nu}} u_2 \, d \sigma
    + \int_{\partial \Omega} \overline{g} (\Lambda_{\mu_2,\varepsilon_2})[g] \, d \sigma
    - \int_{\partial \Omega} \f{1}{\mu_1} \overline{\f{\partial u_2}{\partial \nu}} u_2 \, d \sigma  \\
    &=& \int_{\partial \Omega} \overline{g}
    \left(\Lambda_{\mu_1,\varepsilon_1} - \Lambda_{\mu_2,\varepsilon_2}\right)
    [g] \, d \sigma\,,
\eqnx
which gives the identity
\beqn
    & &
    \int_{\Omega}  \f{1}{\mu_1}  | \nabla(u_2 - u_1) |^2 dx
    -  \omega^2 \int_{\Omega} \varepsilon_1  | u_2 - u_1 |^2 dx
    +  \int_{\Omega} \left( \f{1}{\mu_2} - \f{1}{\mu_1} \right) | \nabla u_2 |^2 dx
    -  \omega^2 \int_{\Omega} \left( \varepsilon_2 - \varepsilon_1 \right)  | u_2 |^2 dx \notag\\
    &=& \int_{\partial \Omega} \overline{g}
    \left(\Lambda_{\mu_1,\varepsilon_1} - \Lambda_{\mu_2,\varepsilon_2}\right)
    [g] \, d \sigma \, .
    \label{201}
\eqn
Swapping $u_1$ and $u_2$ in the above identity implies
\beqn
    & &
    \int_{\Omega}  \f{1}{\mu_2}  | \nabla(u_1 - u_2) |^2 dx
    -  \omega^2 \int_{\Omega} \varepsilon_2  | u_1 - u_2 |^2 dx
    +  \int_{\Omega} \left( \f{1}{\mu_1} - \f{1}{\mu_2} \right) | \nabla u_1 |^2 dx
    -  \omega^2 \int_{\Omega} \left( \varepsilon_1 - \varepsilon_2 \right)  | u_1 |^2 dx \notag\\
    &=& \int_{\partial \Omega} \overline{g}
    \left(\Lambda_{\mu_2,\varepsilon_2} - \Lambda_{\mu_1,\varepsilon_1}\right)
    [g] \, d \sigma \, .
    \label{202}
\eqn
Now (\ref{quadraticpolarization}) follows by subtracting (\ref{201}) from (\ref{202}).
\end{proof}

By the same arguments as those in \cite{abboud} (see also \cite{stability}), we can derive
the following estimate.
\begin{Lemma}
The difference between the interior NtD maps $\Lambda_{\mu,\varepsilon}$ and $\Lambda_{\mu^\delta,\varepsilon^\delta}$
can be represented in terms of
the differences between two sets of electromagnetic parameters $(\mu,\varepsilon)$ and  $(\mu^\delta,\varepsilon^\delta)$:
\beqn
||\Lambda_{\mu^\delta,\varepsilon^\delta} - \Lambda_{\mu,\varepsilon} || \leq C \bigg( ||\varepsilon^\delta - \varepsilon||_{L^\infty(\Omega)} + \bigg|\bigg| \f{1}{\mu^{\delta}} - \f{1}{\mu} \bigg|\bigg|_{L^\infty(\Omega)} \bigg) \, .
\label{ineq1}
\eqn
\end{Lemma}

\ms
Now we can further our analysis on the difference $W_{nm}^{\delta} - W_{nm}$ in terms
of $\varepsilon^\delta - \varepsilon$ and ${1}/{\mu^{\delta}} - {1}/{\mu}$ using (\ref{B_bilinear_useful}) and (\ref{ineq1}).
Recalling $\psi_n$ and $\psi_m^\delta$ from (\ref{recall1}) and (\ref{recall2}),
we can define the solutions $u_m, u_m^\gamma, u_n^\delta$ and $u_n^{\delta \gamma}$ to the following
four systems:
\beqn
    &&\nabla \cdot \f{1}{\mu} \nabla u_m + \varepsilon\omega^2 u_m = 0  \text{ in } \Omega \,; \quad
    \f{1}{\mu} \f{\p}{\p \nu} u_m = \psi_m  \text{ on } \partial \Omega \,;
    \\
    &&\nabla \cdot \f{1}{\mu^\delta} \nabla u_m^\gamma + \varepsilon^\delta \omega^2 u_m^\gamma = 0   \text{ in } \Omega \,; \quad
    \f{1}{\mu^\delta} \f{\p}{\p \nu} u_m^\gamma = \psi_m  \text{ on } \partial \Omega \,;
   \\
    &&\nabla \cdot \f{1}{\mu} \nabla u_n^{\delta \gamma} + \varepsilon\omega^2 u_n^{\delta \gamma}= 0   \text{ in } \Omega \,; \quad
    \f{1}{\mu} \f{\p}{\p \nu} u_n^{\delta \gamma} = \psi^\delta_n \text{ on } \partial \Omega \,;
   \\
    &&\nabla \cdot \f{1}{\mu^\delta} \nabla u_n^\delta + \varepsilon^\delta \omega^2 u_n^\delta= 0   \text{ in } \Omega \,; \quad
    \f{1}{\mu^\delta} \f{\p}{\p \nu} u_n^\delta = \psi^\delta_n \text{ on } \partial \Omega \,.
\eqn
Noting from (\ref{total_u}) that $\psi_n$ and $\psi_m^\delta$ are the density functions in the Neumann potential along $\partial \Omega$ with coefficients $(\mu,\varepsilon)$ and $(\mu^\delta,\varepsilon^\delta)$ respectively,
the solutions $u_m$ and $u_n^\delta$ solve (\ref{scattering1})-(\ref{sommerfield}) with coefficients $(\mu,\varepsilon)$ and $(\mu^\delta,\varepsilon^\delta)$ and the incident field $(u_0)_m$ and $(u_0)_n$ defined as in (\ref{incidence}).
For convenience, we introduce a bilinear form:
\beqn
    B(p , q) := \int_{\partial \Omega} \overline{p} \left(\Lambda_{\mu,\varepsilon} - \Lambda_{\mu^\delta,\varepsilon^\delta}\right) [q] d \sigma  \quad  \forall\,p, q \in H^{-\f{1}{2}}(\p \Om)\,.
    \label{def_B_bilinear}
\eqn
Then (\ref{quadraticpolarization}) gives us an explicit expression of $B(g ,g)$ for $g \in H^{-\f{1}{2}}(\p \Om)$.
By (\ref{B_bilinear_useful}),
the difference $W_{nm}^\delta - W_{nm}$ can be split using the bilinear form $B$ as
\beqn
W_{nm}^\delta - W_{nm}
&=& \mu_0 B(\psi_m , \psi^\delta_n) \notag \\
&=&
\f{\mu_0}{2} \left[
B(\psi_m +  \psi^\delta_n, \psi_m +  \psi^\delta_n)
- B(\psi_m , \psi_m )
- B(\psi^\delta_n,  \psi^\delta_n)
\right] \notag \\
&& +
\f{i \mu_0}{2} \left[
B(\psi_m + i \psi^\delta_n, \psi_m + i \psi^\delta_n)
- B(\psi_m , \psi_m )
- B(\psi^\delta_n,  \psi^\delta_n)
\right] \notag \\
&:=& \f{\mu_0}{2} \I + \f{i \mu_0}{2} \II\, ,
\label{hahaintegral}
\eqn
where $\I$ and $\II$ are given by
\beqn
    \I  &:=& B(\psi_m +  \psi^\delta_n, \psi_m +  \psi^\delta_n) - B(\psi_m , \psi_m )- B(\psi^\delta_n,  \psi^\delta_n) \, , \\
    \II &:=& B(\psi_m + i \psi^\delta_n, \psi_m + i \psi^\delta_n) - B(\psi_m , \psi_m ) - B(\psi^\delta_n,  \psi^\delta_n) \, .
\eqn
By direct calculations, we get the following expression of the term $\I$:
\beqn
\I&=& B(\psi_m +  \psi^\delta_n, \psi_m +  \psi^\delta_n)
- B(\psi_m , \psi_m )
- B(\psi^\delta_n,  \psi^\delta_n) \notag \\
&=&
\f{1}{2} \int_{\Omega} \left(\f{1}{\mu^{\delta}} - \f{1}{\mu}\right)  \left( -| \nabla (u_m + u_n^{\delta \gamma} - u_m^\gamma - u_n^\delta) |^2 + | \nabla (u_m + u_n^{\delta \gamma}) |^2 + | \nabla (u_m^\gamma + u_n^\delta) |^2\right) dx \notag \\
&&
- \f{1}{2} \omega^2 \int_{\Omega} \left(\varepsilon^\delta - \varepsilon\right) \left( -| u_m + u_n^{\delta \gamma} - u_m^\gamma - u_n^\delta |^2 + | u_m + u_n^{\delta \gamma} |^2 + | u_m^\gamma + u_n^\delta |^2\right) dx
\notag \\
&&
- \left[
\f{1}{2} \int_{\Omega} \left(\f{1}{\mu^{\delta}} - \f{1}{\mu}\right)  \left( -| \nabla (u_m  - u_m^\gamma ) |^2 + | \nabla u_m  |^2 + | \nabla u_m^\gamma |^2\right) dx\right.\nb\\
&&{\,}\q\q -\left.\f{1}{2} \omega^2 \int_{\Omega} \left(\varepsilon^\delta - \varepsilon\right) \left( -| u_m  - u_2 |^2 + | u_m  |^2 + | u_m^\gamma  |^2\right) dx
\right]
\notag \\
&&
- \left[
\f{1}{2} \int_{\Omega} \left(\f{1}{\mu^{\delta}} - \f{1}{\mu}\right)  \left( -| \nabla (u_n^{\delta \gamma} - u_n^\delta) |^2 + | \nabla u_n^{\delta \gamma} |^2 + | \nabla u_n^\delta |^2\right) dx\right.\nb\\
&&{\,}\q\q -\left. \f{1}{2} \omega^2 \int_{\Omega} \left(\varepsilon^\delta - \varepsilon\right) \left( -| u_n^{\delta \gamma} - u_n^\delta |^2 + | u_n^{\delta \gamma} |^2 + | u_n^\delta |^2\right) dx
\right] \,. \notag
\label{hahareal0}
\eqn
From (\ref{ineq1}), we get
\beqn
||u_m  - u_m^\gamma  ||^2_{H^1(\Omega)} = O (\widehat{\varepsilon} ) \quad \text{ and }
||u_n^{\delta \gamma} - u_n^\delta  ||^2_{H^1(\Omega)} = O (\widehat{\varepsilon} ) \, ,
\label{haharealbound}
\eqn
where $\widehat{\varepsilon}$ is defined as in (\ref{orderepsilon}). 
Then using (\ref{haharealbound}), we further the estimate of the term $\I$ as follows:
{\small
\beqn
\I& = &
\int_{\Omega} \left(\f{1}{\mu^{\delta}} - \f{1}{\mu}\right)  \left( | \nabla (u_m + u_n^\delta ) |^2 - | \nabla u_m |^2 - | \nabla u_n  |^2\right) dx\nb\\
&&- \omega^2 \int_{\Omega} \left(\varepsilon^\delta - \varepsilon\right) \left( | u_m + u_n^\delta |^2 - | u_m |^2  - | u_n |^2\right) dx   + O (\widehat{\varepsilon}^2 ) \notag \\
& = & 2 \, \text{Re} \left[
 \int_{\Omega} \left(\f{1}{\mu^{\delta}} - \f{1}{\mu}\right) \overline{\nabla u_n^\delta} \cdot \nabla u_m \, dx
-  \omega^2 \int_{\Omega} \left(\varepsilon^\delta - \varepsilon\right) \overline{ u_n^\delta} u_m \, dx \right] + O (\widehat{\varepsilon}^2 )\,.
\label{hahareal}
\eqn
}
Similarly, we can derive the following estimate for the term $\II$:
\beqn
\II & = & B(\psi_m + i \psi^\delta_n, \psi_m + i \psi^\delta_n)
- B(\psi_m , \psi_m )
- B(\psi^\delta_n,  \psi^\delta_n) \notag \\
& = & 2 \, \text{Im} \left[
 \int_{\Omega} \left(\f{1}{\mu^{\delta}} - \f{1}{\mu}\right) \overline{\nabla u_n^\delta} \cdot \nabla u_m \, dx
-  \omega^2 \int_{\Omega} \left(\varepsilon^\delta - \varepsilon\right)  \overline{ u_n^\delta} u_m \, dx \right] + O (\widehat{\varepsilon}^2 ) \, .
\label{hahaimag}
\eqn
Substituting (\ref{hahareal}) and (\ref{hahaimag}) in (\ref{hahaintegral}) gives
\beqn
W_{nm}^\delta - W_{nm} &=&
\mu_0 \int_{\Omega} \left(\f{1}{\mu^{\delta}} - \f{1}{\mu}\right) \overline{\nabla u_n^\delta} \cdot \nabla u_m \, dx
- \mu_0 \omega^2 \int_{\Omega} \left(\varepsilon^\delta - \varepsilon\right)  \overline{ u_n^\delta} u_m \, dx + O (\widehat{\varepsilon}^2 ) \, .
\label{perturbation_final0}
\eqn
Furthermore, we have from (\ref{eq:ufromwnm}) that for all $m \in \mathbb{Z}$,
\beqn
    u_m - (u_0)_m = -\f{i}{4} \sum_{m} H^{(1)}_n (k_0 |x|) e^{i n \theta_x} W_{nm}\, , \q 
    u_m^\delta - (u_0)_m = -\f{i}{4} \sum_{m} H^{(1)}_n (k_0 |x|) e^{i n \theta_x} W_{nm}^\delta\, , \label{eq:ufromwnm_diff2}
\eqn
subtracting the first one from the second in (\ref{eq:ufromwnm_diff2}) gives
\beqn
    u_m^\delta = u_m  -\f{i}{4} \sum_{m} H^{(1)}_n (k_0 |x|) e^{i n \theta_x} \left( W_{nm}^\delta - W_{nm}\right)\,.
    \label{eq:ufromwnm_diff3}
\eqn
Now replacing $u_m^\delta$ in (\ref{perturbation_final0}) by (\ref{eq:ufromwnm_diff3}),
we arrive at the following theorem.

\begin{Theorem}
\label{perturbation}
Assume $(\mu,\varepsilon)$ and  $(\mu^\delta,\varepsilon^\delta)$ are two different sets of electromagnetic parameters,
and $W_{nm}$ and $W_{nm}^{\delta}$ are defined as in (\ref{perturbed}).  Let $\widehat{\varepsilon}$ be defined as in (\ref{orderepsilon}). 
For any $m \in \mathbb{Z}$, let $u_m$
be the solution to (\ref{scattering1})-(\ref{sommerfield}) with the coefficients $(\mu,\varepsilon)$ and
the incident field $(u_0)_m$ in (\ref{incidence}). Then the following estimate holds
for any $n, m \in \mathbb{Z}$:
\beqn
W_{nm}^\delta - W_{nm} = \mu_0 \int_{\Omega} \left(\f{1}{\mu^{\delta}} - \f{1}{\mu}\right) \overline{\nabla u_n } \cdot \nabla u_m
-  \omega^2 \mu_0 \int_{\Omega} \left(\varepsilon^\delta - \varepsilon\right)   \overline{ u_n } u_m  + O (\widehat{\varepsilon}^2 ) \, .
\eqn
\end{Theorem}
The above formula provides a sensitivity analysis in terms of
electromagnetic parameters $(\mu,\varepsilon)$ for arbitrary medium domains $\Om$.
In order to derive reconstruction formulas for $\mu$ and $\varepsilon$ from the scattering coefficients, we shall achieve more explicit and detailed sensitivity analysis and representation formulas
for scattering coefficients $W_{nm}$ when the medium domains are of some special geometry.
This is our focus in the next section.

\section{Explicit reconstruction formulas in the linearized case}  \label{sec7}

For a given $\widehat{\varepsilon} > 0$,
consider $\mu, \varepsilon$ such that $(|| \varepsilon - \varepsilon_0 ||_{L^\infty(\Omega)}^2 + \big|\big|{\mu}^{-1} - {\mu_0^{-1}}\big|\big|^2_{L^\infty(\Omega)})^{1/2}
= \widehat{\varepsilon}$.
Then it follows from (\ref{new_representation_Wnm}) and the definition of $(u_0)_m$ in (\ref{incidence}) that
\beqn
W_{nm}
& =&
 \mu_0 \int_{\Omega} \left( \f{1}{\mu(y)} - \f{1}{\mu_0} \right) \nabla ( J_n(k_0|y|) e^{- i n \theta_y} ) \cdot  \nabla ( J_m(k_0|y|) e^{i m \theta_y} ) dy \notag \\
& & - \omega^2 \mu_0  \int_{\Omega} \left( \varepsilon(y) - \varepsilon_0 \right)  J_n(k_0|y|) J_m(k_0|y|) e^{ i (m - n) \theta_y} dy + O (\widehat{\varepsilon}^2 ) \, . \l{eq:wnm}
\eqn

Now for all $n \neq 0$, we have by direct computing
\beqn
    \p_{x_1} (J_n(k r) e^{i n \theta}) &=& (\cos \theta \p_{r} - \f{\sin \theta}{r} \p_{\theta} ) (J_n(k r) e^{i n \theta}) \notag\\
    &=& \f{k}{2} \cos \theta   ( J_{n-1}(k r) - J_{n+1}(k r)) e^{i n \theta}
    - \f{i n \sin \theta}{r} J_n(k r) e^{i n \theta} \, ,\\
    \p_{x_2} (J_n(k r) e^{i n \theta}) &=& (\sin \theta \p_{r} + \f{\cos \theta}{r} \p_{\theta} ) (J_n(k r) e^{i n \theta}) \notag\\
    &=& \f{k}{2} \sin \theta   ( J_{n-1}(k r) - J_{n+1}(k r)) e^{i n \theta}
    + \f{i n \cos \theta}{r} J_n(k r) e^{i n \theta} \, ,
\eqn
which implies the explicit expression for the gradient term in (\ref{eq:wnm}) for $n ,m \neq 0$:
{\small
\beqn
&& \nabla ( J_n(k_0|y|) e^{- i n \theta_y} ) \cdot \nabla ( J_m(k_0|y|) e^{i m \theta_y} ) \notag \\
& = & \left[ \f{k_0^2}{4} \left( J_{n-1}(k_0 |y|) - J_{n+1}(k_0 |y|)\right) \left( J_{m-1}(k_0 |y|) - J_{m+1}(k_0 |y|)\right)
    + \f{n m}{|y|^2} J_n(k_0 |y|)J_m(k_0 |y|)  \right] e^{i (m-n) \theta_y} \, .
\eqn
}
For $n = 0$, we have $J_0' = - J_1$ and
\beqn
    \p_{x_1} (J_0(k r)) =  - k \cos \theta   ( J_1 (k r))
    \, , \q
    \p_{x_2} (J_0(k r)) =  - k \sin \theta   ( J_1 (k r))
    \,,
\eqn
which yields the following explicit expressions for the gradient term for $n = 0$ or $m =0$:
{\small
\beqn
\nabla ( J_0(k_0|y|) ) \cdot \nabla ( J_m(k_0|y|) e^{i m \theta_y} )
& = & \left[ - \f{k_0^2}{2} \left( J_{1}(k_0 |y|)\right) \left( J_{m-1}(k_0 |y|) - J_{m+1}(k_0 |y|)\right)
      \right] e^{i m \theta_y} \, , \\
\nabla ( J_n(k_0|y|) e^{-i n \theta_y} ) \cdot \nabla ( J_0(k_0|y|) )
& = & \left[ - \f{k_0^2}{2} \left( J_{n-1}(k_0 |y|) - J_{n+1}(k_0 |y|)\right) \left( J_{1}(k_0 |y|) \right)
      \right] e^{ -i n \theta_y} \, , \\
\nabla ( J_0(k_0|y|) ) \cdot \nabla ( J_0(k_0|y|) )
& = & k_0^2 \left( J_{1}(k_0 |y|) \right)^2
      \, .
\eqn
}
These explicit formulas lead us to the following corollary.
\begin{Corollary}\l{cor:wnm}
Let $(\mu,\varepsilon)$ be a pair of electromagnetic parameters in $\Omega$, and
$\widehat{\varepsilon}= (|| \varepsilon - \varepsilon_0 ||^2_{L^\infty} + \big|\big|{\mu}^{-1} - {\mu_0^{-1}}\big|\big|^2_{L^\infty})^{1/2}
$. Then the scattering coefficients $W_{nm}[ \, \varepsilon, \mu, \omega, \Omega \,]$
admit the following expansions:
\beqn
W_{nm}
& =&
\f{\mu_0 k_0^2}{4} \int_{\Omega} \left( \f{1}{\mu(y)} - \f{1}{\mu_0} \right)\left( J_{n-1}(k_0 |y|) - J_{n+1}(k_0 |y|)\right) \left( J_{m-1}(k_0 |y|) - J_{m+1}(k_0 |y|)\right)
 e^{i (m-n) \theta_y} \, dy \notag \\
& & + \mu_0 n m \int_{\Omega} \left( \f{1}{\mu(y)} - \f{1}{\mu_0} \right) \f{1}{|y|^2} J_n(k_0 |y|)J_m(k_0 |y|) e^{i (m-n) \theta_y} \, dy \notag \\
& & - \omega^2 \mu_0  \int_{\Omega} \left( \varepsilon(y) - \varepsilon_0 \right)  J_n(k_0|y|) J_m(k_0|y|) e^{ i (m - n) \theta_y} dy + O (\widehat{\varepsilon}^2 ) \q (\m{for} ~~n,m\ne 0)
\label{wnm_explicit} \\
W_{00}
& =&
\mu_0 k_0^2  \int_{\Omega} \left( \f{1}{\mu(y)} - \f{1}{\mu_0} \right)  \left( J_{1}(k_0 |y|) \right)^2 \, dy  - \omega^2 \mu_0  \int_{\Omega} \left( \varepsilon(y) - \varepsilon_0 \right)  \left( J_{0}(k_0 |y|) \right)^2 \, dy + O (\widehat{\varepsilon}^2 )
\label{wnm_explicit3}\\
W_{n0}
& =&- \f{\mu_0 k_0^2}{2} \int_{\Omega} \left( \f{1}{\mu(y)} - \f{1}{\mu_0} \right)   \left( J_{n-1}(k_0 |y|) - J_{n+1}(k_0 |y|)\right) \left( J_{1}(k_0 |y|) \right) e^{ -i n \theta_y} \, dy \notag \\
& & - \omega^2 \mu_0  \int_{\Omega} \left( \varepsilon(y) - \varepsilon_0 \right)  J_n(k_0|y|) J_0(k_0|y|) e^{ -i n \theta_y} \, dy + O (\widehat{\varepsilon}^2 ) \, \q (\m{for} ~~n\ne 0)\,.
\label{wnm_explicit2}
\eqn
\end{Corollary}

By means of the asymptotic behavior (\ref{decayhaha})
and the estimates in Corollary\,\ref{cor:wnm}, we obtain the following estimate for all $n, m \in \mathbb{Z}$:
\beqn
| W_{nm} |
& \leq &
\f{\mu_0^2 \varepsilon_0 \omega^2}{4}
 \bigg|\bigg| \f{1}{\mu} - \f{1}{\mu_0} \bigg|\bigg|_{L^\infty(\Omega)} \f{C^{|n|+ |m|-2 }}{{(|n|-1)}^{(|n|-1)} {(|m|-1)}^{(|m|-1)}}  \notag \\
& & + \mu_0 \bigg|\bigg| \f{1}{\mu} - \f{1}{\mu_0} \bigg|\bigg|_{L^\infty(\Omega)}\f{C^{|n|+|m|-2 }}{{|n|}^{|n|-1}{|m|}^{|m|-1}} + \omega^2 \mu_0 || \varepsilon - \varepsilon_0 ||_{L^\infty(\Omega)}  \f{C^{|n|+|m|}}{{|n|}^{|n|}{|m|}^{|m|}} + C \widehat{\varepsilon}^2 \, .
\l{eq:estimatew}
\eqn
Moreover, by comparing (\ref{wnm_explicit}) with (\ref{decayhaha}) for large $m$ and $n$,
we can see that the two integrals with the term $({\mu}^{-1} - {\mu_0^{-1}})$ dominate.
This suggests that we may separate the effect of $({\mu}^{-1} - {\mu_0^{-1}})$ and $\varepsilon - \varepsilon_0$ on $W_{nm}$ and recover $\mu$ and $\varepsilon$ alternatively: First use the scattering coefficients $W_{nm}$ for large $m,n$ to recover $\mu$,
then use the scattering coefficients $W_{nm}$ for small $m,n$ to recover $\varepsilon$.
Furthermore, with the integral expression (\ref{wnm_explicit})
we may work out each term explicitly for some special domains, e.g., $\Omega = B_R(0)$.
For simplicity, we will present our detailed derivations and calculations for
the special case with $\mu = \mu_0$ but $\varepsilon \neq \varepsilon_0$
in the remainder of this section, though most of the conclusions can be extended to the general case
with $\mu \neq \mu_0$ and $\varepsilon \neq \varepsilon_0$.
It is easy to see for the special domain $\Omega = B_R(0)$ and the special case
with $\mu = \mu_0$ but $\varepsilon \neq \varepsilon_0$ that
$W_{nm}$ are simplified to be
\beqn
 W_{nm} &=& -  \omega^2 \mu_0  \int_{0}^{R} \int_{0}^{2 \pi} \left(\varepsilon(y) - \varepsilon_0\right)  J_n(k_0r_y) J_m(k_0r_y) e^{ i (m - n) \theta_y} r_y dr_y d \theta_y + O (\widehat{\varepsilon}^2 )  \, ,
 \label{wnm_epsilon}
\eqn
where $y = (r_y, \theta_y)$ is the polar coordinate.

\subsection{Radially symmetric case}
In this subsection we derive formulas
to recover the electromagnetic parameter $\varepsilon$ from the scattering coefficients $W_{nm}$
in the case with $\mu = \mu_0$, but $\varepsilon \neq \varepsilon_0$ with $\varepsilon$ being radially symmetric in
$\Omega = B_R(0)$. We shall write $\widehat{\varepsilon} := || \varepsilon - \varepsilon_0||_{L^\infty(\Omega)} $, and
$\varepsilon(y) = \varepsilon(r_y)$. It is straightforward to see from (\ref{wnm_epsilon}) that
\beqn
 W_{nm} =
  -  2 \pi \omega^2 \mu_0   \int_{0}^{R} \left(\varepsilon(r_y) - \varepsilon_0\right)  [J_n(k_0r_y)]^2 r_y dr_y + O (\widehat{\varepsilon}^2 )
  ~\text{ for ~} m = n ~~and
  ~~O (\widehat{\varepsilon}^2 ) ~\text{ for ~} m \neq n.
   \label{wnm_epsilon_radial}
\eqn
It follows readily from (\ref{wnm_explicit}), (\ref{wnm_explicit2}) and (\ref{wnm_explicit3}) that
the same conclusion as in (\ref{wnm_epsilon_radial}) for $m \neq n$ can be obtained
for the more general case when $\mu \neq \mu_0$ and $\varepsilon \neq \varepsilon_0$,
provided that both $\mu$ and $\varepsilon$ are radial symmetric in $\Omega = B_R(0)$.

In the later part of this subsection, we shall establish an explicit formula for
computing the electromagnetic parameter $\varepsilon$ in terms of the scattering coefficients
$W_{nn}(k) := W_{nn}[ \, \varepsilon, \mu, \omega(k), \Omega \, ]$, where
$
 \omega(k)= {k}/{\sqrt{\varepsilon_0 \mu_0}}
 \label{frequency}
 $
is the frequency depending on $k \in \mathbb{R}^+$.
For the sake of convenience, we define the following coefficient
\beqn
\mathcal{H}_{n}^{(0)}:= \int_0^{\infty} \f{W_{nn}(k)}{k} \, dk\,.
\label{H_def1}
\eqn
Using the following orthogonality of the Bessel functions $\{ J_{n}( r k) \}_{r >0}$ for a given $n \in \mathbb{Z}$:
\beqn
    \int^{\infty}_{0} J_{n}(r k) J_{n}(r' k) k \, dk = \f{\delta(r-r')}{r} \, \quad \forall\,r, r' >0 \, ,
    \label{orthogonal}
\eqn
we obtain from (\ref{wnm_epsilon}) and (\ref{frequency}) that
\beqnx
 \mathcal{H}_{n}^{(0)} = \int_0^{\infty} \f{W_{nn}(k)}{k} \, dk &=&
 - \f{2 \pi}{\varepsilon_0 } \int_{0}^{R} \left(\varepsilon(r_y) - \varepsilon_0\right) \left( \int_0^\infty  J_n(k r_y) J_n(k r_y)k dk \right) r_y dr_y + O (\widehat{\varepsilon}^2 ) \\
 &=&  - \f{2 \pi}{\varepsilon_0 }  \int_{0}^{R} \left(\varepsilon(r_y) - \varepsilon_0\right) dr_y + O (\widehat{\varepsilon}^2 ) \, ,
\eqnx
which gives the average of $ \varepsilon(r_y) - \varepsilon_0 $ along the radial direction.  Next, we shall extend the above observation to obtain more information about $\varepsilon$. This motivates us with the following definition.
\begin{Definition}
For $n \in \mathbb{Z}$, let $W_{nn}(k) := W_{nn}[ \, \varepsilon, \mu, \omega(k), \Omega \, ]$ be defined as in (\ref{eq:w_nm})
with $\omega(k)= {k}/{\sqrt{\varepsilon_0 \mu_0}}$.  For $l, n \in \mathbb{Z}$ and $l \geq 0$, let $g^{(l)}_{n}(k)$ be functions such that
\beqn
    \int_0^\infty g^{(l)}_{n}(k) J_n(k r) J_n(k r) k^2 \, dk = r^{l-1} \, \q \forall\, r > 0.
    \label{moment}
\eqn
Then we define the coefficients $\mathcal{H}_{n}^{(l)}$  by
\beqn
 \mathcal{H}_{n}^{(l)}:= \int_0^{\infty} g^{(l)}_{n}(k) W_{nn}(k) \,dk \q \forall\,~ l, n \in \mathbb{Z}, ~l \geq 0\,.
 \label{H_def2}
\eqn
\end{Definition}
We will show the existence of functions $ g^{(l)}_{n}$ satisfying
(\ref{moment}) and derive their explicit expressions in Appendix \ref{appendixB}.

We see from the orthogonality relation (\ref{orthogonal}) that $g^{(0)}_{n}(k) = {1}/{k}$. Thus the definition of $\mathcal{H}_{n}^{(l)}$ in (\ref{H_def2}) is consistent with (\ref{H_def1}) for $l = 0$. With this definition, we are able to recover the $l$-th moment of $\varepsilon(r_y) - \varepsilon_0$ from the scattering coefficients $W_{nn}(k)$ measured at different wavenumber $k$ but for one fixed $n \in \mathbb{Z}$.  Putting (\ref{wnm_epsilon}), (\ref{moment}) into (\ref{H_def2}), we get
\beqnx
 \mathcal{H}_{n}^{(l)}= \int_0^{\infty} g^{(l)}_{n}(k) W_{nn}(k) \,dk &=&
 - \f{2 \pi}{\varepsilon_0 }  \int_{0}^{R} \left(\varepsilon(r_y) - \varepsilon_0\right) \left( \int_0^\infty  g^{(l)}_{n}(k) J_n(k r_y) J_n(k r_y) k^2 dk \right) r_y dr_y + O (\widehat{\varepsilon}^2 ) \\
 &=&  - \f{2 \pi}{\varepsilon_0 } \int_{0}^{R} r_y^{l} \left(\varepsilon(r_y) - \varepsilon_0\right) dr_y + O (\widehat{\varepsilon}^2 ) \, .
\eqnx
By this relation, the electromagnetic coefficient $\varepsilon$ can be reconstructed explicitly.
\begin{Corollary}
\label{recover1}
Let $\Omega = B_R(0)$ be the disk of center $0$ and radius $R$. Let $(\mu,\varepsilon)$ be the pair of electromagnetic parameters in $\Omega$ and $(\mu_0,\varepsilon_0)$ be the parameters of the homogeneous background. Assume that the parameters satisfy $\mu=\mu_0$ and $\varepsilon$ is radially symmetric, i.e., $\varepsilon(y) = \varepsilon(r_y)$, and  $\widehat{\varepsilon} = || \varepsilon - \varepsilon_0||_{L^\infty(\Omega)}$. Then the coefficients $\mathcal{H}_{n}^{(l)}$ defined in (\ref{H_def2}) satisfy the following relationship for $l, n \in \mathbb{Z}$ and $l \geq 0$,
\beqn
 \mathcal{H}_{n}^{(l)} = - \f{2 \pi}{\varepsilon_0 } \int_{0}^{R} r_y^{l} \left(\varepsilon(r_y) - \varepsilon_0\right) dr_y + O (\widehat{\varepsilon}^2 ) \, .
 \label{Mellin}
\eqn
For $\alpha \in \mathbb{Z}$, the $\alpha$-th Fourier coefficient $\mathfrak{F}_{r_y}\left[\varepsilon(r_y) - \varepsilon_0\right]( \alpha )$
of $\varepsilon(r_y) - \varepsilon_0$ can be written explicitly by
\beqn
\mathfrak{F}_{r_y}\left[\varepsilon(r_y) - \varepsilon_0\right]( \alpha ) = - \f{2 \pi}{\varepsilon_0 } \sum_{l=0}^{\infty} \f{( - \f{2 \pi}{R} i \alpha )^{l}}{l!} \mathcal{H}_{n}^{(l)} + O (\widehat{\varepsilon}^2 ) \,
\label{recover_05}
\eqn
for a fixed $n \in \mathbb{Z}$, and the electromagnetic coefficient $\varepsilon$ can be explicitly expressed as, for a fixed $n \in \mathbb{Z}$,
\beqn
(\varepsilon - \varepsilon_0)(r_y) = - \f{2 \pi}{\varepsilon_0 }\sum_{ \alpha =-\infty}^{\infty} \sum_{l=0}^{\infty} e^{i \f{2 \pi}{R} \alpha r_y  } \f{( - \f{2 \pi}{R} i \alpha )^{l}}{l !} \mathcal{H}_{n}^{(l)} + O (\widehat{\varepsilon}^2 ) \, .
\label{recover_formula1}
\eqn
\end{Corollary}
We remark that, with (\ref{recover_formula1}), we are able to reconstruct $\varepsilon$ from a set of scattering coefficients $\{ W_{nn}(k) | k \in \mathbb{R}^+ \}$ for all wavenumbers $k$ but with only a fixed $n \in \mathbb{Z}$ . Choosing $n$ small yields a stable reconstruction of $\varepsilon$ from far-field patterns at frequencies $k \in [0, k_{\mathrm{max}}]$ by approximating
$\mathcal{H}_n^{(l)}$ with $\int_0^{k_{\mathrm{max}}} g^{(l)}_{n}(k) W_{nn}(k) \,dk$ and truncating the infinite sums in (\ref{recover_formula1}).

\subsection{Angularly symmetric case}
In this subsection we would like to recover the electromagnetic parameter $\varepsilon$ from the scattering coefficients $W_{nm}$
for the special domain $\Omega = B_R(0)$ and the special case when
$\mu = \mu_0$ and the electromagnetic coefficient $\varepsilon$ only depends on $\theta_y$, i.e., $\varepsilon(y) = \varepsilon(\theta_y)$. Directly from (\ref{wnm_epsilon}), we have, for $n,m \in \mathbb{Z}$,
\beqn
 W_{nm} &=& -  \omega^2 \mu_0  \left( \int_0^{2 \pi} \left(\varepsilon(\theta_y) - \varepsilon_0\right) e^{i(m-n) \theta_y} d \theta_y \right)  \left( \int_{0}^{R} J_n(k_0r_y) J_m(k_0r_y)  r_y dr_y \right) + O (\widehat{\varepsilon}^2 ) \notag\\
  &=& -  \omega^2 \mu_0 C_{k_0} (m,n)  \int_0^{2 \pi} \left(\varepsilon(\theta_y) - \varepsilon_0\right) e^{i(m-n) \theta_y} d \theta_y  + O (\widehat{\varepsilon}^2 ) \, .
 \label{wnm_epsilon_angular}
\eqn
where $C_{k_0} (m,n)$ is given by
\beqn
C_{k_0} (m,n) & : =& \int_{0}^{R} J_n(k_0r_y) J_m(k_0r_y)  r_y \, dr_y \, , \quad n,m \in \mathbb{Z}.
\eqn
We can see that for $n,m \in \mathbb{Z}$, $C_{k_0} (m,n)$ actually satisfies
\beqn
C_{k_0} (m,n) & : =& \int_{0}^{R} J_n(k_0r_y) J_m(k_0r_y)  r_y dr_y \notag \\
& =&  \f{1}{4 \pi^2 }\int_{0}^{2 \pi}  \int_{0}^{2 \pi} \left[ \int_{0}^{R} e^{i k_0r_y \left( \sin \theta +\sin \phi \right) } r_y dr_y \right]  e^{- i \left( n \theta + m \phi \right) } d \theta d \phi \notag \\
& =&  \f{1}{4 \pi^2 }\int_{0}^{2 \pi}  \int_{0}^{2 \pi}
\left[
\f{ R e^{i k_0 R \left( \sin \theta+\sin\phi \right) } }{i k_0 \left( \sin\theta+\sin\phi \right)}
+
\f{ e^{i k_0 R \left( \sin\theta+\sin(\phi) \right) } }{ k_0^2 \left( \sin\theta+\sin\phi \right)^2}
\right]
e^{- i \left( n \theta + m \phi \right) } d \theta d \phi \notag \\
& =&  \mathfrak{F}_{\theta,\phi}
\left[
\f{ R e^{i k_0 R \left( \sin \theta+\sin\phi \right) } }{i k_0 \left( \sin\theta+\sin\phi \right)}
+
\f{ e^{i k_0 R \left( \sin\theta+\sin(\phi) \right) } }{ k_0^2 \left( \sin\theta+\sin\phi \right)^2}
\right]
(n,m) \, ,
\label{coefficient_fourier}
\eqn
where $\mathfrak{F}_{\theta,\phi}$ stands for the Fourier coefficient in both arguments $\theta$ and $\phi$. Formula (\ref{coefficient_fourier}) indicates that the coefficients $C_{k_0} (m,n)$, $m,n \in \mathbb{Z}$, can be approximated via FFT or calculated explicitly. From (\ref{wnm_epsilon_angular}), we can obtain the Fourier coefficients $\mathfrak{F}_{\theta_y} \left[\varepsilon(\theta_y) - \varepsilon_0\right]$ of $\varepsilon(\theta_y) - \varepsilon_0$
as follows:
\beqnx
  \mathfrak{F}_{\theta_y} \left[\varepsilon(\theta_y) - \varepsilon_0\right](n-m) = - \f{W_{nm}}{ \omega^2 \mu_0 C_{k_0} (m,n)} + O (\varepsilon^2 )  \, ,
\eqnx
for all $n,m \in \mathbb{Z}$.  Thus we have the following corollary.
\begin{Corollary}
\label{recover2}
Let $\Omega = B_R(0)$ and $\widehat{\varepsilon} := || \varepsilon - \varepsilon_0||_{L^\infty(\Omega)}$, and
the same assumptions be assumed for $(\mu,\varepsilon)$ and $(\mu_0,\varepsilon_0)$ as
in Corollary\,\ref{recover1}, except that the radial symmetry of $\varepsilon$ is now replaced
by the angular symmetry, i.e., $\varepsilon(y) = \varepsilon(\theta_y)$.
Then for all $n,m \in \mathbb{Z}$,  the scattering coefficients $W_{nm}$ defined in (\ref{eq:w_nm}) satisfy the following relationship with the Fourier coefficients of $\varepsilon(\theta_y) - \varepsilon_0$:
\beqn
  \mathfrak{F}_{\theta_y} \left[\varepsilon(\theta_y) - \varepsilon_0\right](n-m) = - \f{W_{nm}}{ \omega^2 \mu_0 C_{k_0} (m,n)} + O (\varepsilon^2 )  \, .
  \label{expill}
\eqn
Let $\{ (n_{l}, m_{l}) \}_{l \in \mathbb{Z}} \subset \mathbb{Z} \times \mathbb{Z}$ be such that $n_{l}-m_{l} = l$ for all $l \in \mathbb{Z}$. Then the electromagnetic coefficient $\varepsilon$ can be explicitly expressed by
\beqn
(\varepsilon - \varepsilon_0)(\theta_y) = - \sum_{l=-\infty}^{\infty} \f{W_{n_{l} m_{l}}}{ \omega^2 \mu_0 C_{k_0} (m_{l},n_{l})} e^{i 2 \pi l \theta_y} + O (\widehat{\varepsilon}^2 ) \, .
\label{recover_formula2}
\eqn
\end{Corollary}
\noindent We can see from (\ref{recover_formula2}) that in order to recover the electromagnetic coefficient $\varepsilon$ in the angular symmetric case, we only need to know $\{ W_{n_{l} m_{l}} \}_{l \in \mathbb{Z}}$ where $\{ (n_{l}, m_{l}) \}_{l \in \mathbb{Z}} \subset \mathbb{Z} \times \mathbb{Z}$ is such that $n_{l}-m_{l} = l$ for $l \in \mathbb{Z}$.
So we do not necessarily require all the scattering coefficients $W_{nm}$ to recover $\varepsilon$,
instead we may choose $\{ W_{n_{l} m_{l}} \}_{l \in \mathbb{Z}}$  of any particular  $\{ (n_{l}, m_{l}) \}_{l \in \mathbb{Z}}$, for instance we may fix $n_{l} = 0$. Truncating the
 sum in (\ref{recover_formula2}) up to $N$ gives a stable reconstruction formula (for the low-frequency part) with an angular resolution limit depending on $N$. Higher is $N$ better is the angular resolution. When $W_{nm}$ are corrupted by noise, $N$ can be computed as a function of the signal to noise ratio in the measurements.

\subsection{General case}
In this subsection, we try to derive formulas to recover the parameter $\varepsilon$ from the set of scattering coefficients $\{W_{nm}(k)|n,m \in \mathbb{Z}, k \in \mathbb{R}^+\}$, where $W_{nm}(k) := W_{nm}[ \, \varepsilon, \mu, \omega(k), \Omega \, ]$ is defined in (\ref{eq:w_nm}) with $\omega(k)$ satisfying (\ref{frequency}) when $\Omega = B_R(0)$, $\mu = \mu_0$, without any assumption on the parameter $\varepsilon$.  We would like to combine the ideas in the proofs of (\ref{recover_formula1}) and (\ref{recover_formula2}) to get a more general result. Now we start with a general $\varepsilon$ which admits the Fourier expansion:
\beqn
\varepsilon(r_y,\theta_y) - \varepsilon_0 = \sum_{\alpha \in \mathbb{Z} } \mathfrak{F}_{\theta_y} \left[ \varepsilon(r_y,\theta_y) - \varepsilon_0 \right](\alpha) e^{i \alpha \theta_y},
\label{fourier_epsilon}
\eqn
where $\mathfrak{F}_{\theta_y} \left[\varepsilon(r_y,\theta_y) - \varepsilon_0 \right](\alpha)$ is the $\alpha$-th Fourier coefficient with respect to $\theta_y$ fixing $r_y$. Then we plug the expansion (\ref{fourier_epsilon}) into (\ref{wnm_epsilon}) to get
\beqn
 W_{nm} &=& -  2 \pi \omega^2 \mu_0   \int_{0}^{R} \mathfrak{F}_{\theta_y} \left[ \varepsilon(r_y,\theta_y) - \varepsilon_0 \right](n-m) J_n(k_0r_y) J_m(k_0r_y)  r_y dr_y + O (\widehat{\varepsilon}^2 )  \, .
 \label{finaldev}
\eqn
Following the definition of $\mathcal{H}_{n}^{(l)}$ in (\ref{H_def2}), we define a generalized coefficient $\mathcal{H}^{(l)}_{nm}$ below.
\begin{Definition}
For $n, m \in \mathbb{Z}$, let $W_{nm}(k) := W_{nm}[ \, \varepsilon, \mu, \omega(k), \Omega \, ]$ be defined as in (\ref{eq:w_nm}) where $\omega(k)$ is defined as in (\ref{frequency}). For $l, n, m \in \mathbb{Z}$ and $l \geq 0$, let $g^{(l)}_{nm}(k)$ be functions such that
\beqn
    \int_0^\infty g^{(l)}_{nm}(k) J_n(k r) J_m(k r) k^2 \, dk = r^{l-1} \, ,
    \label{moment2}
\eqn
for any $r > 0$. Then the coefficients $\mathcal{H}_{nm}^{(l)}$ are defined as, for $l, n, m \in \mathbb{Z}$ and $l \geq 0$,
\beqn
 \mathcal{H}_{nm}^{(l)}:= \int_0^{\infty}g^{(l)}_{nm}(k) W_{nm}(k) \,dk .
 \label{H_def3}
\eqn
\end{Definition}
We refer to Appendix \ref{appendixC} for the existence of functions  $g^{(l)}_{nm}$ satisfying
(\ref{moment2}).

 With this definition, we are able to recover, for all $n,m\in \mathbb{Z}$, the $l$-th moment of the Fourier coefficients $\mathfrak{F}_{\theta_y} \left[ \varepsilon(r_y,\theta_y)  - \varepsilon_0 \right](n-m)$ with respect to $r_y$ from the scattering coefficients $W_{nm}(k)$ measured at different frequencies $k$ . Actually, we have, putting (\ref{frequency}), (\ref{moment2}) and (\ref{finaldev}) into (\ref{H_def2}),
\beqnx
 \mathcal{H}_{nm}^{(l)}&=& \int_0^{\infty} g^{(l)}_{nm}(k) W_{nm}(k) \,dk \\
 &=&
 -  \f{2 \pi}{\varepsilon_0} \int_{0}^{R} \mathfrak{F}_{\theta_y} \left[ \varepsilon(r_y,\theta_y) - \varepsilon_0 \right](n-m) \left( \int_0^\infty g^{(l)}_{nm}(k) J_n(k r_y) J_m(k r_y) k^2 \, dk \right) r_y dr_y + O (\widehat{\varepsilon}^2 ) \\
 &=&  -  \f{2 \pi}{\varepsilon_0}  \int_{0}^{R} r_y^{l} \mathfrak{F}_{\theta_y} \left[ \varepsilon(r_y,\theta_y) - \varepsilon_0 \right](n-m) dr_y + O (\widehat{\varepsilon}^2 ) \, ,
\eqnx
for all $n,m\in \mathbb{Z}$.  Therefore, similar to (\ref{recover_05}), we get, for all $n,m,\alpha \in \mathbb{Z}$,
\beqnx
\mathfrak{F}_{r_y,\theta_y}  \left[\varepsilon(r_y,\theta_y) - \varepsilon_0\right](\alpha , n-m) = -  \f{2 \pi}{\varepsilon_0} \sum_{l=0}^{\infty} \f{( - \f{2 \pi}{R} i \alpha )^{l}}{l !} \mathcal{H}_{nm}^{(l)} + O (\widehat{\varepsilon}^2 ) \, .
\eqnx
Fixing a set $\{ (n_{p}, m_{p}) \}_{p \in \mathbb{Z}} \subset \mathbb{Z} \times \mathbb{Z}$ such that $n_{p}-m_{p} = p$ for $p \in \mathbb{Z}$, we are able to recover $\varepsilon - \varepsilon_0$ explicitly expressed as
\beqnx
\varepsilon - \varepsilon_0
= -  \f{2 \pi}{\varepsilon_0} \sum_{ \alpha =-\infty}^{\infty} \sum_{ p =-\infty}^{\infty} \sum_{l=0}^{\infty} e^{ i \left[l \theta_y + \f{2 \pi}{R} \alpha r_y\right] } \f{( - \f{2 \pi}{R} i \alpha )^{l}}{l!} \mathcal{H}_{n_{p}m_{p}}^{(l)} + O (\widehat{\varepsilon}^2 )  \, .
\eqnx
\begin{Corollary}
\label{recover3}
Let $\Omega = B_R(0)$ and $\widehat{\varepsilon} := || \varepsilon - \varepsilon_0||_{L^\infty(\Omega)}$, and
the same assumptions be assumed for $(\mu,\varepsilon)$ and $(\mu_0,\varepsilon_0)$ as
in Corollary\,\ref{recover1}, except that the radial symmetry of $\varepsilon$ is now replaced by
the Fourier expansion (\ref{fourier_epsilon}).
Then for $l, n, m  \in \mathbb{Z}$ and $l \geq 0$,
the coefficients $ \mathcal{H}_{nm}^{(l)}$ defined in (\ref{H_def3}) satisfy the following relationship:
\beqn
 \mathcal{H}_{nm}^{(l)} = -  \f{2 \pi}{\varepsilon_0}  \int_{0}^{R} r_y^{l} \mathfrak{F}_{\theta_y} \left[ \varepsilon(r_y,\theta_y) - \varepsilon_0 \right](n-m) dr_y + O (\widehat{\varepsilon}^2 ) \, .
 \label{Mellin2}
\eqn
Moreover, for all $n,m,\alpha \in \mathbb{Z}$, the $(\alpha , n-m)$-th Fourier coefficient of $\varepsilon - \varepsilon_0$
can be written explicitly by
\beqn
\mathfrak{F}_{r_y,\theta_y}  \left[\varepsilon(r_y,\theta_y) - \varepsilon_0\right](\alpha , n-m) = -  \f{2 \pi}{\varepsilon_0} \sum_{l=0}^{\infty} \f{( - \f{2 \pi}{R} i \alpha )^{l}}{l !} \mathcal{H}_{nm}^{(l)} + O (\widehat{\varepsilon}^2 ) \, .
\label{recover_25}
\eqn
Let $\{ (n_{p}, m_{p}) \}_{p \in \mathbb{Z}} \subset \mathbb{Z} \times \mathbb{Z}$ be such that $n_{p}-m_{p} = p$ for all $p \in \mathbb{Z}$, then the electromagnetic coefficient $\varepsilon$ can be explicitly expressed by

\beqn
(\varepsilon - \varepsilon_0)(r_y, \theta_y)
= -  \f{2 \pi}{\varepsilon_0} \sum_{ \alpha =-\infty}^{\infty} \sum_{ p =-\infty}^{\infty} \sum_{l=0}^{\infty} e^{ i \left[p \theta_y + \f{2 \pi}{R} \alpha r_y\right] } \f{( - \f{2 \pi}{R} i \alpha )^{l}}{l !} \mathcal{H}_{n_{p}m_{p}}^{(l)} + O (\widehat{\varepsilon}^2 ) \, .
\label{recover_formula3}
\eqn
\end{Corollary}
\noindent We remark that expression (\ref{recover_formula3}) generalizes (\ref{recover_formula1}) and (\ref{recover_formula2}).  Moreover, similar to observations in previous subsections, we can see that in order to recover the electromagnetic coefficient $\varepsilon$, we only need to know $\{ W_{n_{p} m_{p}} (k) | p \in \mathbb{Z}, k \in \mathbb{R}^+ \}$ where $\{ (n_{p}, m_{p}) \}_{p \in \mathbb{Z}} \subset \mathbb{Z} \times \mathbb{Z}$ is such that $n_{p}-m_{p} = p$ for $p \in \mathbb{Z}$. Therefore, we may choose a particular choice $\{ (n_{p}, m_{p}) \}_{p \in \mathbb{Z}}$, for instance we can let $n_{p} = 0$. This tells us that we are able to recover
$\varepsilon$ with incomplete data of the scattering coefficients. As pointed out earlier,
we may truncate the series in (\ref{recover_formula3}) and approximate $\mathcal{H}_{n_{p}m_{p}}^{(l)}$ by
$\int_0^{k_{\mathrm{max}}} g^{(l)}_{n_{p}m_{p}}(k) W_{n_{p}m_{p}}(k) \,dk$.


\section{Concluding remarks}  \label{sec8}
In this paper we have introduced the concept of scattering coefficients for inverse medium scattering problems
in heterogeneous media, and established important properties (such as symmetry and tensorial properties)
of the scattering coefficients as well as their various representations in terms of the NtD maps.
An important relationship between the scattering coefficients and the far-field pattern is also derived.
Furthermore,  the sensitivity of the scattering coefficients with respect to the changes in the permittivity and permeability distributions is explored, which enables us to derive explicit reconstruction formulas for the permittivity and permeability parameters
in the linearized case. These formulas show on one hand the  stability of the reconstruction from multifrequency measurements,
and on the other hand, the exponential instability of the reconstruction from far-field measurements at a fixed frequency.
The scattering coefficient based approach introduced in this work is a new promising direction for solving the long-standing inverse scattering problem with heterogeneous inclusions.
They can be combined with some existing methods such as
the continuation method \cite{bao1, bao2, bao3, coifman} to improve the stability and the resolution of the reconstructed images.

\appendix

\section{Construction of the Neumann function}  \label{appendixA}
In this section we construct the Neumann function $N_{\mu,\varepsilon}$ associated with
\beqn
    -L u := \nabla \cdot \f{1}{\mu} \nabla u + \omega^2 \varepsilon u
    \label{pde}
\eqn in $\Omega$, which is an open connected domain with $\mathcal{C}^2$ boundary in $\mathbb{R}^d$ for $d=2,3$.
We shall also estimate its singularity. Again, we assume that $0$ is not a Neumann eigenvalue of $L$ on $\Omega$.

In order to show the existence of the Neumann function, we first consider the following problem: given $f \in \mathcal{C}^{\infty}_c (\Omega)$, find $u \in H^1(\Omega)$ such that
\beqn
    \nabla \cdot \f{1}{\mu} \nabla u + \omega^2 \varepsilon u = f   \text{ in } \Omega \, ; \quad
    \f{1}{\mu} \f{\partial u }{\partial \nu} = 0  \text{ on } \partial \Omega \, .
    \label{systemtwo}
\eqn
By the well-known De Giorgi-Nash-Moser Theorem \cite{moser} for the $L^\infty$ coefficient and Sobelov embedding, we have for $R>0$ such that $B_{R} \subset \Omega$ that
\beqn
|| u ||_{L^\infty(B_{R/2})} & \leq & C \left(R^{1-\f{d}{2}} || u ||_{L^\f{2d}{d-2} (\Omega)} + R^2 ||f||_{L^\infty(B_{R})} \right) \notag \\
& \leq & C \left(R^{1-\f{d}{2}} || u ||_{H^1(\Omega)} + R^2 ||f||_{L^\infty(B_{R})} \right) \, .
\label{ineq1b}
\eqn
On the other hand, one can prove using the same argument as in \cite{abboud} that for all $f \in \mathcal{C}^{\infty}_c (\Omega)$, there exists a unique $u \in H^1(\Omega)$ satisfying (\ref{systemtwo}) such that
\beqn
|| u ||_{H^1(\Omega)} \leq C || f ||_{H^1(\Omega)} \, .
\label{ineq2}
\eqn
Therefore, combining (\ref{ineq1b}) and (\ref{ineq2}), we have
\beqn
|| u ||_{L^\infty(B_{R/2})} & \leq & C \left(R^{1-\f{d}{2}} || f ||_{H^1(\Omega)} + R^2 ||f||_{L^\infty(B_{R})} \right) \, .
\label{ineq3}
\eqn
Now consider $f \in \mathcal{C}^{\infty}_c (\Omega)$ such that the support of $f$ is contained in
$B_{R} \subset \Omega$ for some $R$. Then for any $\phi \in H^1(\Omega)$, we deduce
by the H\"older inequality and Sobelov embedding that
\beqn
    \bigg| \int_{\Omega} f \phi dx \bigg| \leq || f ||_{L^\f{2d}{d+2} (B_{R})} || \phi ||_{L^\f{2d}{d-2} (\Omega)} \leq C || f ||_{L^\f{2d}{d+2} (B_{R})} || \phi ||_{H^1(\Omega)} \leq C R^{\f{d+2}{2}} ||f||_{L^\infty(B_{R})} || \phi ||_{H^1(\Omega)} \, .
    \label{ineq3point5}
\eqn
For all $u \in H^1(\Omega), \Delta u \in L^2(\Omega)$ with $\f{\partial u }{\partial \nu} = 0$, the following Poincar\'e-type inequality can be shown by contradiction
\beqn
||u||_{H^1(\Omega)}^2 \leq C |\langle L u , u \rangle_{L^2(\Omega)}| \, .
\label{ineq10}
\eqn
Setting $\phi = u \in H^1(\Omega)$ in (\ref{ineq3point5}) and combining it with (\ref{ineq10}), we get
\beqn
   ||u||_{H^1(\Omega)}^2 \leq C |\langle L u , u \rangle_{L^2(\Omega)}| = C \bigg| \int_{\Omega} f u dx \bigg| \leq C R^{\f{d+2}{2}} ||f||_{L^\infty(B_{R})} || u ||_{H^1(\Omega)} \, ,
\eqn
which gives
\beqn
   ||u||_{H^1(\Omega)} \leq C R^{\f{d+2}{2}} ||f||_{L^\infty(B_{R})} \, .
     \label{ineq4}
\eqn
Therefore, combining (\ref{ineq1b}) and (\ref{ineq4}), we have
\beqn
|| u ||_{L^\infty(B_{R/2})} & \leq & C  R^2 ||f||_{L^\infty(B_{R})} \, .
\label{ineq5}
\eqn
This inequality plays a key role in proving the existence of the Neumann function and
and establishing its estimate.

Now we are ready to construct a Neumann function for the system (\ref{pde}), following the technique in \cite{singular}. Fix a function $\varphi \in \mathcal{C}^{\infty}_c (B_1(0))$ and $0 \leq \varphi \leq 2$ such that $ \int_{B_1(0)} \varphi dx = 1 $. Let $y \in \Omega$ be fixed. For any $\varepsilon >0$, we define
\beqn
    \varphi_{\varepsilon,y} (x) = \varepsilon^{-d} \varphi\left(\f{x-y}{\varepsilon}\right)\, .
\eqn
Let $N^{\varepsilon}(\cdot,y) \in H^1(\Omega)$ be the "averaged Neumann function" such that it satisfies (\ref{systemtwo}) with $f = \varphi_{\varepsilon,y} $, then we immediately have from (\ref{ineq5}) that for all $\varepsilon \leq \f{R}{2}$,
\beqn
|| N^{\varepsilon}(\cdot,y) ||_{L^\infty(B_{R/2})} & \leq & C  R^2 ||\varphi_{\varepsilon,y}||_{L^\infty(B_{R})} \, .
\label{ineq6}
\eqn
This $L^\infty$ estimate for $N^{\varepsilon}(\cdot,y)$ can be further improved later.

It is worth mentioning that we have the $H^1$ estimate for $N^{\varepsilon}(\cdot,y)$ by the H\"older inequality and Sobelov embedding. Indeed for all $\phi \in H^1(\Omega)$, we have
\beqn
    \bigg| \int_{\Omega}   \varphi_{\varepsilon,y}  \phi dx \bigg| \leq ||   \varphi_{\varepsilon,y}  ||_{L^\f{2d}{d+2} (B_{\varepsilon})} || \phi ||_{H^1(\Omega)} \leq C \varepsilon^{\f{2-d}{2}}  || \phi ||_{H^1(\Omega)} \, .
    \label{ineq6point5}
\eqn
Setting $\phi = N^{\varepsilon}(\cdot,y) \in H^1(\Omega)$ in (\ref{ineq6point5}) and combining it with (\ref{ineq10}), we yield
\beqn
   ||N^{\varepsilon}(\cdot,y)||_{H^1(\Omega)}^2 \leq C |\langle L N^{\varepsilon}(\cdot,y) , N^{\varepsilon}(\cdot,y) \rangle_{L^2(\Omega)}| = C \bigg| \int_{\Omega}   \varphi_{\varepsilon,y}  N^{\varepsilon}(\cdot,y) dx \bigg| \leq C \varepsilon^{\f{2-d}{2}}  || N^{\varepsilon}(\cdot,y) ||_{H^1(\Omega)}  ,
\eqn
which gives
\beqn
     || N^{\varepsilon}(\cdot,y) ||_{H^{1}(\Omega)} \leq C \varepsilon^{\f{2-d}{2}} \, .
     \label{ineq7}
\eqn

Now we come back to the $L^\infty$ bound for $N^{\varepsilon}(\cdot,y)$. For all $\varepsilon \leq \f{R}{2}$, $R \leq d_y$ where $d_y$ is the distance between $y$ and $\partial \Omega$,
\beqn
   \int_{\Omega} f N^{\varepsilon}(\cdot,y) dx = - \int_{\Omega} L u N^{\varepsilon}(\cdot,y) dx =  - \int_{\Omega} u L N^{\varepsilon}(\cdot,y) dx = - \int_{\Omega} u \varphi_{\varepsilon,y} dx \, ,
\eqn
hence we have from $ \int_{B_1(0)} \varphi dx = 1 $, $\varphi \geq 0$ and (\ref{ineq5}) that
\beqn
   \bigg| \int_{\Omega} f N^{\varepsilon}(\cdot,y) dx \bigg| \leq  || \varphi_{\varepsilon,y} ||_{L^1(B_{R/2})} || u ||_{L^\infty(B_{R/2})} = || u ||_{L^\infty(B_{R/2})} \leq C  R^2 ||f||_{L^\infty(B_{R})} \, .
\eqn
Therefore, by duality, we have the $L^1$ estimate for $N^{\varepsilon}(\cdot,y)$ with $\varepsilon \leq \f{R}{2}$ and $R \leq d_y$ as follows
\beqnx
|| N^{\varepsilon}(\cdot,y)  ||_{L^1(B_{R})} & \leq & C  R^2 \, .
\label{ineq8}
\eqnx

We wish to use De Giorgi-Nash-Moser theorem once again to get a sharp $L^\infty$ estimate for $N^{\varepsilon}(\cdot,y)$ from (\ref{ineq8}) following an idea in \cite{singular}.  Indeed, for any $x \in \Omega$ such that $0 < |x-y| < {d_y}/{2}$, take
$R : = {2|x-y|}/{3}$. Note that if $\varepsilon < {R}/{2}$, then $ N^{\varepsilon}(\cdot,y)  \in H^{1}(B_R(x))$ satisfies $-L N^{\varepsilon}(\cdot,y) = 0$ in $B_R(x)$. For $r \leq \f{R}{3}$, we derive by the De Giorgi-Nash-Moser theorem for
the $L^\infty$ coefficient, we get that
\beqn
|N^{\varepsilon}(x,y)| \leq || N^{\varepsilon}(\cdot,y) ||_{L^\infty(B_r(x))}  \leq C r^{-d} || N^{\varepsilon}(\cdot,y)  ||_{L^1(B_{r}(x))} \leq C r^{-d} || N^{\varepsilon}(\cdot,y)  ||_{L^1(B_{3r}(y))} \leq C r^{2-d} \, .
\label{ineq8b}
\eqn
Therefore we recover the result in \cite{singular} for our operator $L$: for any $x,y \in \Omega$ satisfying $0 < |x-y| < {/d_y}{2}$, we have
\beqn
|N^{\varepsilon}(x,y)| \leq C |x-y|^{2-d} \quad \forall \,\varepsilon < \f{|x-y|}{3} \,.
\label{usefulestimate}
\eqn

Next, we would like to show the weak convergence of a subsequence of $N^{\varepsilon}(\cdot,y)$ in $W^{1,p} (B_r(y))$ and $H^{1} (\Omega \backslash B_r(y))$.  For this purpose, we need to have a uniform bound of $N^{\varepsilon}(\cdot,y)$ in such norms with respect to $\varepsilon$.  We shall proceed as in \cite{singular}. First for a $r \leq {d_y}/{2}$,
we get directly from (\ref{ineq6}) for $\varepsilon \geq {r}/{6}$ that
\beqn
     || \nabla N^{\varepsilon}(\cdot,y) ||_{L^2(\Omega \backslash B_r(y))} \leq || \nabla N^{\varepsilon}(\cdot,y) ||_{L^2(\Omega)} \leq C r^{\f{2-d}{2}} \, .
     \label{ineq9}
\eqn
For $\varepsilon < {r}/{6}$, we wish to control the gradient of $N^{\varepsilon}(\cdot,y)$ by $N^{\varepsilon}(\cdot,y)$
outside the ball $B_r(y)$ and establish a similar estimate as the Caccioppoli's inequality inside the ball $B_r(y)$.
To do so, we introduce a smooth function $\eta$ on $\mathbb{R}^d$ satisfying
\beqn
0 \leq \eta \leq 1, \quad  |\nabla \eta| \leq \f{4}{r}\,, \q
\eta \equiv 1 \text{ in }\mathbb{R}^d \backslash B_r(y), \quad \eta \equiv 0 \text{ in } B_\f{r}{2}(y)\,.
\label{testfunction}
\eqn
Using (\ref{ineq10}) and the properties (\ref{testfunction}), we can deduce
\beqn
&&||N^{\varepsilon}(\cdot,y) \eta ||_{H^1(\Omega)}^2\nb\\
&=&\int_{\Omega} \left( |\nabla N^{\varepsilon}(\cdot,y)|^2 + |N^{\varepsilon}(\cdot,y)|^2 \right) \eta^2 dx + 2 \int_{\Omega} N^{\varepsilon}(\cdot,y) \eta \nabla N^{\varepsilon}(\cdot,y) \cdot \nabla \eta   dx +  \int_{\Omega} |\nabla \eta|^2 \eta^2 dx \notag \\
&\leq& C |\langle L (N^{\varepsilon} \eta ) , N^{\varepsilon}(\cdot,y) \eta  \rangle_{L^2(\Omega)}| \notag \\
&=& C \left|\int_{\Omega} \left( \f{1}{\mu} |\nabla N^{\varepsilon}(\cdot,y)  \eta |^2 - \varepsilon \omega^2 |N^{\varepsilon}(\cdot,y)|^2  \eta^2 \right) dx \right| \notag \\
&=& C \left| \int_{\Omega} \left( \f{1}{\mu} |\nabla N^{\varepsilon}(\cdot,y)|^2 - \epsilon \omega^2 |N^{\varepsilon}(\cdot,y)|^2 \right) \eta^2 dx + 2 \int_{\Omega} \f{1}{\mu} N^{\varepsilon}(\cdot,y) \eta \nabla N^{\varepsilon}(\cdot,y) \cdot \nabla \eta   dx +  \int_{\Omega}  \f{1}{\mu}  |\nabla \eta|^2 \eta^2 dx \right| \notag \\
&=& C \left| \int_{\Omega} \left( \f{1}{\mu} \nabla N^{\varepsilon}(\cdot,y) \cdot \nabla \left( N^{\varepsilon}(\cdot,y) \eta^2\right) - \epsilon \omega^2 |N^{\varepsilon}(\cdot,y)|^2 \eta^2 \right) dx  +  \int_{\Omega} \f{1}{\mu}  |\nabla \eta|^2  \eta^2 dx \right| \notag \\
&=& C \left| \int_{\Omega}  L N^{\varepsilon}(\cdot,y) \left( N^{\varepsilon}(\cdot,y) \eta^2\right) dx  +  \int_{\Omega} \f{1}{\mu}  |\nabla \eta|^2  \eta^2 dx \right| \notag \\
&=& C \left| \int_{\Omega}  \varphi_{\varepsilon,y} N^{\varepsilon}(\cdot,y) \eta^2 dx  +  \int_{\Omega} \f{1}{\mu} |\nabla \eta|^2  \eta^2 dx \right| \notag \\
&\leq& C  \int_{\Omega}  |\nabla \eta|^2  \eta^2 dx\,. \notag
\eqn
From this and the Cauchy-Schwarz's inequality it follows that
\beqn
 \int_{\Omega} \left( |\nabla N^{\varepsilon}(\cdot,y)|^2 + |N^{\varepsilon}(\cdot,y)|^2 \right) \eta^2 dx
\leq C  \int_{\Omega} |\nabla \eta|^2  \eta^2 dx + \f{1}{2} \int_{\Omega} |\nabla N^{\varepsilon}(\cdot,y)|^2 \eta^2 + 4 \int_{\Omega} |N^{\varepsilon}(\cdot,y)|^2 |\nabla \eta|^2 |dx \, , \notag
\eqn
which implies
\beqn
|| \nabla N^{\varepsilon}(\cdot,y)||_{L^2(\Omega \backslash B_r(y))}^2 \leq C \left( \int_{\Omega} |\nabla \eta|^2  \eta^2 + |N^{\varepsilon}(\cdot,y)|^2 |\nabla \eta|^2 dx \right) \, .
\eqn
Now we have from (\ref{usefulestimate}) and (\ref{testfunction}) that for $\varepsilon < \f{r}{6}$,
\beqn
|| \nabla N^{\varepsilon}(\cdot,y)||_{L^2(\Omega \backslash B_r(y))}^2 &\leq& C \left( \int_{B_r(y) \backslash B_{\f{r}{2}}(y)} |\nabla \eta|^2  \eta^2 + |N^{\varepsilon}(\cdot,y)|^2 |\nabla \eta|^2 |dx \right) \,  \notag \\
&\leq& C r^{-2} \left( \int_{B_r(y) \backslash B_{\f{r}{2}}(y)} 1  + |x-y|^{2(2-d)} dx \right) \,  \notag \\
&\leq& C r^{-2} \left( \int_{\f{r}{2}}^{r} \left(1  + |t|^{2(2-d)}\right) t^{d-1} dt \right) \,  \notag \\
&\leq& C r^{2-d} .
\label{ineq11}
\eqn
Combining (\ref{ineq9}) and (\ref{ineq11}), we have
\beqn
     || \nabla N^{\varepsilon}(\cdot,y) ||_{L^2(\Omega \backslash B_r(y))} \leq C r^{\f{2-d}{2}} \,
     \quad \forall \,r \in \left(0 ,\f{d_y}{2}\right) , \, \,\varepsilon > 0 \, .
     \label{combineineq1}
\eqn
Moreover it follows directly from (\ref{ineq9}) that for all $\epsilon < \f{r}{6}$,
\beqn
     || N^{\varepsilon}(\cdot,y) ||_{L^\f{2d}{d-2} (\Omega \backslash B_r(y))} \leq C r^{\f{2-d}{2}},
     \label{ineq12}
\eqn
while for $\epsilon \geq \f{r}{6}$,
\beqn
     || N^{\varepsilon}(\cdot,y) ||_{L^\f{2d}{d-2} (\Omega \backslash B_r(y))} \leq C || N^{\varepsilon}(\cdot,y) ||_{H^1(\Omega)} \leq C r^{\f{2-d}{2}}.
     \label{ineq13}
\eqn
Now the combination (\ref{ineq12}) with (\ref{ineq13}) yields
\beqn
     || N^{\varepsilon}(\cdot,y) ||_{L^\f{2d}{d-2}(\Omega \backslash B_r(y))} \leq C r^{\f{2-d}{2}} \,,   \quad \forall\, r \in \left(0 ,\f{d_y}{2}\right) , \, \,\varepsilon > 0 \, .
     \label{combineineq2}
\eqn
On the other hand, the following estimate comes from (\ref{combineineq1}) and (\ref{combineineq2})
for all $r \in \left(0 ,d_y\right)$ that
\beqn
     || N^{\varepsilon}(\cdot,y) ||_{L^\f{2d}{d-2}(\Omega \backslash B_r(y))} + || \nabla N^{\varepsilon}(\cdot,y) ||_{L^2(\Omega \backslash B_r(y))} \leq C r^{\f{2-d}{2}} \,,  \quad \forall \varepsilon > 0 \, .
     \label{totalcombineineq}
\eqn
With this estimate (\ref{totalcombineineq}), we can readily derive the following estimate for
$r \in \left(0 ,d_y\right)$ by following the same argument as in \cite{singular}:
\beqn
     || N^{\varepsilon}(\cdot,y) ||_{L^p(B_r(y))} \leq C r^{2 - d + \f{d}{p}} \,,  \quad \forall \varepsilon > 0 \,,  \quad \forall p \in \big[1 , \f{d}{d-2}\big) \, ,  \\
     \label{totalcombineineq2}
     || \nabla N^{\varepsilon}(\cdot,y) ||_{L^p(B_r(y))} \leq C r^{1 - d + \f{d}{p}} \,, \quad \forall \varepsilon > 0 \,,  \quad \forall p \in \big[1 , \f{d}{d-1}\big) \, .
     \label{totalcombineineq3}
\eqn

Now the same argument as  in \cite{singular} will ensure the existence of a sequence
$\{\varepsilon_n\}_{n=1}^\infty$ going to zero and a function $N(\cdot,y)$ such that $N^{\varepsilon_n}(\cdot,y)$ converges to $N(\cdot,y)$ weakly in $W^{1,p}(B_r(y))$ for $1 < p < \f{d}{d-1}$ and weakly in $H^1(\Omega \backslash B_r(y))$ for all $r \in \left(0 , d_y \right)$. It is then routine (see \cite{singular}) to get an estimate of $N(\cdot,y)$ from (\ref{totalcombineineq})
for all $r \in \left(0 , d_y\right)$,
\beqn
     || N(\cdot,y) ||_{L^\f{2d}{d-2}(\Omega \backslash B_r(y))} + || \nabla N(\cdot,y) ||_{L^2(\Omega \backslash B_r(y))} \leq C r^{\f{2-d}{2}} ,
     \label{totalcombineineqlimit}
\eqn
and from (\ref{totalcombineineq2}) and (\ref{totalcombineineq3}) for all $r \in \left(0 , d_y\right)$ that
\beqn
     || N(\cdot,y) ||_{L^p(B_r(y))} \leq C r^{2 - d + \f{d}{p}} \,,    \quad \forall p \in \big[1 , \f{d}{d-2}\big) \, ,  \\
     \label{totalcombineineq2limit}
     || \nabla N(\cdot,y) ||_{L^p(B_r(y))} \leq C r^{1 - d + \f{d}{p}} \,, \quad \forall p \in \big[1 , \f{d}{d-1}\big) \, .
     \label{totalcombineineq3limit}
\eqn

Our section ends with the pointwise estimate for $N(x,y)$ by using De Giorgi-Nash-Moser theorem once again. For any $x \in \Omega$ such that $0 < |x-y| < \f{d_y}{2}$, take $R : = \f{2|x-y|}{3}$. From (\ref{totalcombineineqlimit}) we have $N(\cdot,y)\in H^1(B_R(x))$ satisfying $-L N(\cdot,y)=0$ in $B_R(x)$. Then by De Giorgi-Nash-Moser theorem for the $L^\infty$ coefficient, we can deduce
the following estimate with the same technique as in (\ref{ineq8b}),
\beqn
|N(x,y)| \leq C r^{-d} || N(\cdot,y)  ||_{L^1(B_{r}(x))} \leq C r^{-d} || N(\cdot,y)  ||_{L^1(B_{3r}(y))} \leq C |x-y|^{2-d} \, .
\label{totalcombineineq4limit}
\eqn
This gives the estimate of the singularity type as $x$ approaches to $y$.

\section{Existence of functions $g^{(l)}_n$}  \label{appendixB}
In this section we wish to show the existence of functions $g^{(l)}_n$ satisfying (\ref{moment})
for all $l, n \in \mathbb{N}$ and provide their explicit expressions. From the fact that
\beqn
[J_n(kr)]^2 = \f{2}{\pi} \int_0^{\pi/2} J_{2n} (2 kr \sin \phi) d \phi \,
\label{integral}
\eqn
for all $n \in \mathbb{N}$, we substitute (\ref{integral}) into (\ref{moment}) to get, for all $l , n \in \mathbb{N}$, that
\beqn
    \f{2}{\pi} \int_0^{\pi/2} \int_0^\infty g^{(l)}_{n}(k)  J_{2n} (2 kr \sin \phi) k^2 \, dk  d \phi= r^{l-1} \, \q \forall\, r > 0 \, ,
    \label{moment2b}
\eqn
Recall the following orthogonal relationship for Hankel functions
\beqn
    \int_0^\infty  J_{2n} (kr)J_{2n}(k'r ) r \, dr = \f{\delta(k - k')}{k'} \, .
    \label{ortho}
\eqn
for all $k , k' > 0$ and $n \in \mathbb{N}$,.  Now, for $l, n \in \mathbb{N}$, consider the Hankel tranform of $r^{l-1}$ of order $2n$ at $p >0$,
\beqn
    [\mathcal{H}_{2n} (r^{l-1}) ](p)&:=& \int_0^{\infty} r^{l-1}  J_{2n}(r p ) r \, dr \, .
\eqn
By a change of variables, we have
\beqn
    [\mathcal{H}_{2n} (r^{l-1})  ](p)&=& \int_0^{\infty} r^{l-1}  J_{2n}(r p ) r \, dr \notag \\
    &=& \f{2}{\pi} \int_0^{\pi/2} \int_0^{\infty}  g^{(l)}_{n}(k) \left( \int_0^\infty  J_{2n} (2 kr \sin \phi)J_{2n}(r p ) r \, dr \right) k^2 \, dk  d \phi \notag \\
    &=&  \f{1}{\pi} \int_0^{\infty} \int_{\phi = 0}^{\phi = \pi/2}  k  g^{(l)}_{n}(k) \f{\left( \int_0^\infty  J_{2n} (2 kr \sin \phi)J_{2n}(r p ) r \, dr \right)}{\cos \phi} \, d ( 2 k \sin\phi ) \, dk \notag \\
    &=& \f{1}{\pi} \int_0^{\infty} \int_{0}^{2k} k  g^{(l)}_{n}(k) \f{\left( \int_0^\infty  J_{2n} (r l )J_{2n}(r p ) r \, dr \right)}{\sqrt{1 - (\f{l}{2 k})^2}}  \, d l \, dk .
    \label{eqil1}
\eqn
From orthogonality relation (\ref{ortho}), we get that from (\ref{eqil1}) that
\beqn
    [\mathcal{H}_{2n} (r^{l-1})  ](p)
    &=& \f{1}{\pi} \int_0^{\infty} \chi_{\{p<2k\}}(k) \f{k  g^{(l)}_{n}(k)}{ p \sqrt{1 - (\f{p}{2 k})^2}}\, dk \notag \\
    &=& \f{1}{p \pi} \int_{\f{p}{2}}^{\infty} \f{k^2  g^{(l)}_{n}(k)}{\sqrt{k^2 - (\f{p}{2})^2}}\, dk  \, .
    \label{eqil2}
\eqn
Therefore,  for $p >0$, we have
\beqn
    - 2p \, [\mathcal{H}_{2n} (r^{l-1})  ](2 p)
    &=& - \f{1}{\pi} \int_{p}^{\infty} \f{k^2  g^{(l)}_{n}(k)}{\sqrt{k^2 - p^2}}\, dk  \, .
    \label{eqil3}
\eqn
Now we recall that the Abel transform of an integrable function $f(r)$ defined on $r\in (0,\infty)$ is as follows
\beqn
    F(y) := [\mathcal{A} (f)](y) :=
    2 \int_{y}^{\infty} \f{f(r)r}{ \sqrt{r^2 - y^2} }\, dr \, , \quad y \in (0,\infty) \, ,
    \label{abel}
\eqn
whenever the above integral is well-defined. If $f(r) = O(\f{1}{r})$ as $r \rightarrow \infty$, then its inverse Abel transform is well-defined and $f$ satisfies the following
\beqn
    f(r) = [\mathcal{A}^{-1} (F)](r) :=
    - \f{1}{\pi} \int_{r}^{\infty} \f{F'(y)}{ \sqrt{y^2 - r^2} }\, dy \, , \quad r \in (0,\infty) \, .
    \label{invabel}
\eqn
Comparing (\ref{invabel}) and (\ref{eqil3}), we can see that, for all $l , n \in \mathbb{N}$, the functions
\beqn
    G^{(l)}_{n}(p) : = - 2p \, [\mathcal{H}_{2n} (r^{l-1})  ](2 p) \, , \quad p \in (0,\infty)
    \label{defGin}
\eqn
are nothing but the inverse Abel transform of a primitive function of $k^2  g^{(l)}_{n}(k)$. Therefore, applying Abel transform to both sides of the equation (\ref{eqil3}) and then differentiating with respect to the argument of the function, we get
\beqn
     \f{\p }{\p k} \left[\mathcal{A} \left(G^{(l)}_{n} \right) \right] (k)
    = k^2  g^{(l)}_{n}(k)  \, .
    \label{eqil4}
\eqn
Consequently,  we have the following explicit expression for $g^{(l)}_{n}$
\beqn
    g^{(l)}_{n}(k) =  \f{1}{k^2} \f{\p }{\p k} \left[\mathcal{A} \left(G^{(l)}_{n} \right) \right] (k) \, , \quad k \in (0,\infty),
    \label{answer}
\eqn
where $G^{(l)}_{n}$ is defined as in (\ref{defGin}). One can see by direct substitution of (\ref{answer}) back into (\ref{moment}) that the functions $g^{(l)}_{n}$ defined as (\ref{answer}) satisfy equation (\ref{moment}). Therefore,  we have shown existence of functions satisfying (\ref{moment}).

\section{Existence of functions $g^{(l)}_{nm}$}  \label{appendixC}
In this section we show the existence of functions $g^{(l)}_{nm}$ for $l,n,m \in \mathbb{Z}$ and $l\geq 0$ which satisfies
(\ref{moment2}), namely the integral equation
\beqn
    \int_0^\infty g^{(l)}_{nm}(k) J_n(k r) J_m(k r) k^2 \, dk = r^{l-1} \, \q \forall\, r > 0 \, .
    \label{generalmoment}
\eqn
For this purpose, we would like to first investigate the following integral, which will be useful in the subsequent discussion.  For $n,m \in \mathbb{N}$ and $p \in \mathbb{C}$ such that $m + n > Re(p)>0$, we consider the following integral,
\beqn
A_{nm}(p) := \int_0^\infty J_n(x)J_m(x)x^{-p} \, dx\,, \quad p \in \mathbb{C}\, , \quad m + n > Re(p)>0 \, .
\label{integralAnm}
\eqn
We observe that the function $A_{nm}: \{p \in \mathbb{C}\,: \, m + n > Re(p)>0 \}\rightarrow \mathbb{C}$ is a holomorphic function on the strip $\{p \in \mathbb{C}\,: \, a<Re(p)<b \}$ for some $a, b \in \mathbb{R}$ such that $a < b$.  This comes from the fact that for $n,m \in \mathbb{N}$ and $p \in \mathbb{C}$ such that $m + n > Re(p)>0$, the integral $A_{nm}(p)$ defined in (\ref{integralAnm}) can be expressed in the following form,
\beqn
A_{nm}(p)=\int_0^\infty J_m(x)J_n(x)x^{-p} \, dx
=\frac{2^{- p} \, \Gamma(p) \, \Gamma\left(\f{1 + m + n - p}{2}\right)}{
   \Gamma\left(\f{1 + m - n + p}{2}\right) \Gamma\left(\f{1 - m + n + p}{2}\right)\Gamma\left(\f{1 + m + n + p}{2}\right)\,.
   }
\label{integral_explicit}
\eqn
Now given $a, b \in \mathbb{R}$ and $s \in \mathbb{C}$ such that $a <Re(s)<b$, we recall the definition of the Mellin tranform of an integrable function $f(r)$ defined for $r\in (0,\infty)$:
\beqn
[\mathcal{M}(f)](y) := \int_0^\infty r^{s-1} f(r) dr\, , \quad  a <Re(s)<b
\eqn
whenever the above integral is well-defined. With $a,b \in \mathbb{R}$, we write the function $\zeta_{a,b}$ as
\beqn
\zeta_{a,b}(x) = x^{-a} \, \text{ for } 0 < x \leq 1\, , \quad \text{ and } x^{-b} \, \text{ for } 1 < x < \infty .
\eqn
We define the linear space $\mu_{a,b}(0,\infty)$ as the space of all infinitely smooth compactly supported complex valued functions $\phi \in C_c^\infty(0,\infty)$ for which
\beqn
   || \phi ||_{k,\zeta_{a,b},K} := \sup_{K} |\zeta_{a,b}(x) x^{k+1} D_x^k \phi (x)|
\eqn
is finite for all $k \in \mathbb{N}$ and for any compact set $K \Subset (0,\infty)$.  Consider an increasing sequence of compact sets $\{K_n \Subset (0,\infty)\}_{n\in \mathbb{N}}$ such that $\bigcup_{n\in \mathbb{N}} K_n = (0,\infty)$, the countable norms $|| \cdot ||_{k,\zeta_{a,b},K_n} \, , k,n \in \mathbb{N}$ gives a topology on $\mu_{a,b}(0,\infty)$ such that $\mu_{a,b}(0,\infty)$ becomes a complete locally convex space.  We define the dual of $\mu_{a,b}(0,\infty)$, $\mu_{a,b}'(0,\infty)$, and equip it with the weak topology. With these definitions at hand, the Mellin transform can be naturally extended to the space $\mu_{a,b}'(0,\infty)$, see \cite{alomari,Zemanian} for more details.  We denote the generalized Mellin transform also as $\mathcal{M}$.

Now from (\ref{generalmoment}) and (\ref{integralAnm}), we have
for all $l,n,m \in \mathbb{Z}$ with $l\geq 0$ and $p\in \mathbb{C}$ such that $Re(p)>l$,
\beqn
\left[\mathcal{M}(1)\right] (l-p)&=& \int_0^\infty \int_0^\infty g^{(l)}_{nm}(k) J_n(k r) J_m(k r) r^{-p} k^2 \, dk  \,dr \notag \\
&=& \int_0^\infty g^{(l)}_{nm} (k) \left( \int_{r=0}^\infty  J_n(k r) J_m(k r) (kr)^{-p} \,d(kr) \right) k^{p + 1} \, dk   \notag \\
&=& \int_0^\infty g^{(l)}_{nm} (k) \left( \int_{0}^\infty  J_n(r) J_m(r) r^{-p} \,d r \right) k^{p + 1} \, dk   \notag \\
&=& A_{nm}(p) \left[ \mathcal{M} ( g^{(l)}_{nm}) \right] (p+2)\, ,
\eqn
where $A_{nm}(p)$ is known explicitly as (\ref{integral_explicit}). Therefore we get, for all $l,n,m \in \mathbb{Z}$ with $l\geq 0$ and $p\in \mathbb{C}$ such that $Re(p)>l$,
\beqn
\left[ \mathcal{M} ( g^{(l)}_{nm}) \right] (p+2) = \f{\left[\mathcal{M}(1)\right] (l-p)}{A_{nm}(p)}\,,
\eqn
then the existence of $g^{(l)}_{nm}$ is ensured by the Mellin inverse transform.

\end{document}